\documentclass[12pt,oneside]{amsart}
\usepackage[foot]{amsaddr}
\usepackage[utf8]{inputenc}
\usepackage[T1]{fontenc}
\usepackage[BCOR=0mm]{typearea}
\usepackage{amsmath,amssymb,amsaddr,mathtools}
\usepackage{graphicx}
\usepackage[only,llbracket,rrbracket]{stmaryrd}
\usepackage{enumerate}
\mathtoolsset{centercolon,showonlyrefs,showmanualtags}
\usepackage[usenames,dvipsnames,svgnames]{xcolor}
\usepackage{hyperref}
\usepackage{nicefrac}
\usepackage{lmodern}
\usepackage{microtype}
\usepackage{dsfont}
\usepackage[normalem]{ulem}

\usepackage{color}
\allowdisplaybreaks[4]

\def\XXint#1#2#3{{\setbox0=\hbox{$#1{#2#3}{\int}$ }
\vcenter{\hbox{$#2#3$ }}\kern-.59\wd0}}

\newcommand{\edge}{{\scriptstyle\mid}}
\newcommand{\grad}{\nabla}
\renewcommand{\div}{\grad\mkern-3mu\cdot\mkern-2mu}
\newcommand{\laplace}{\Delta}

\newcommand{\N}{\mathbb{N}}
\newcommand{\cH}{\ensuremath{\mathcal H}}

\newcommand{\F}{{\mathcal F}}

\newcommand{\R}{\mathbb{R}}

\DeclareMathOperator{\ac}{ac}

\DeclareMathOperator{\diam}{diam}

\DeclareMathOperator{\iso}{iso}
\DeclareMathOperator{\Lip}{Lip}

\newcommand{\cD}{\ensuremath{\mathcal{D}}}

\newcommand{\cP}{\mathcal{P}}

\DeclareMathAlphabet{\mathup}{OT1}{\familydefault}{m}{n}
\newcommand{\dx}[1]{\mathop{}\!\mathup{d} #1}
\newcommand{\dt}{\mathop{}\!\delta t}
\newcommand{\pderiv}[3][]{\frac{\mathop{}\!\mathup{d}^{#1} #2}{\mathop{}\!\mathup{d} #3^{#1}}}

\newcommand{\eps}{\varepsilon}

\newcommand{\dist}{\text{\rm dist}}

\newtheorem{proposition}{Proposition}
\newtheorem*{proposition*}{Proposition}
\newtheorem{theorem}{Theorem}
\newtheorem{lemma}{Lemma}
\newtheorem{remark}{Remark}

\DeclarePairedDelimiter{\abs}{\lvert}{\rvert}
\DeclarePairedDelimiter{\norm}{\lVert}{\rVert}
\DeclarePairedDelimiter{\bra}{(}{)}

\DeclarePairedDelimiter{\set}{\{}{\}}

\newcommand{\cE}{\mathcal{E}}
\newcommand{\cF}{\mathcal{F}}
\newcommand{\cS}{\mathcal{S}}
\newcommand{\cT}{\mathcal{T}}

\definecolor{darkblue}{rgb}{0,0,0.6}
\hypersetup{
    pdftitle={The Scharfetter-Gummel scheme for aggregation-diffusion equations},
    pdfauthor={André Schlichting and Christian Seis},
    colorlinks=true,    
    linkcolor=darkblue, 
    citecolor=darkblue, 
    filecolor=darkblue, 
    urlcolor=darkblue   
}

\begin{document}

\title{The Scharfetter--Gummel scheme for aggregation-diffusion equations}

\author{Andr\'e Schlichting\textsuperscript{*}}
\email{a.schlichting@uni-muenster.de}
\author{Christian Seis\textsuperscript{*}}
\address[*]{Institut f\"ur Analysis und Numerik, Universit\"at M\"unster, 48149 M\"unster, Germany.}
\email{seis@uni-muenster.de}

\date{\today}

\begin{abstract}
In this paper, we propose a finite-volume scheme for aggregation-diffusion equations based on a Scharfetter--Gummel approximation of the quadratic, nonlocal flux term. This scheme is analyzed concerning well-posedness and convergence towards solutions to the continuous problem. Also, it is proven that the numerical scheme has several structure-preserving features. More specifically, it is shown that the discrete solutions satisfy a free-energy dissipation relation analogous to the continuous model. Consequently, the numerical solutions converge in the large time limit to stationary solutions, for which we provide a thermodynamic characterization.
Numerical experiments complement the study.
\end{abstract}
\subjclass[2010]{}

\maketitle


\section{Introduction}

\subsection{The aggregation-diffusion equation}

We consider the aggregation-diffusion equation in a smoothly bounded domain $\Omega$ in $\R^d$, describing the evolution of a probability density $\rho \in \cP(\Omega)$ according to
\begin{equation}\label{e:agg-diff}
  \partial_t \rho = \nabla\cdot \bra*{ \kappa \nabla \rho + \rho \nabla W \ast \rho} \:,
\end{equation}
where $\kappa>0$ is the diffusion constant and $W: \R^d \to \R$ an interaction potential. The convolution in \eqref{e:agg-diff} is to be understood in $\Omega$, that is
\[
\bra*{W\ast\rho}(x) = \int_{\Omega} W(x-y)\rho(y)\dx{y}.
\]
Moreover, in order to preserve the total mass, we have to impose the no-flux boundary condition
\begin{equation}
\label{bc}
\kappa \partial_{\nu}\rho + \rho\partial_{\nu} W\ast \rho=  0\qquad\text{on }\partial \Omega,
\end{equation}
where $\nu$ denotes the outer normal vector on $\partial \Omega$. The initial datum will be denoted by~$\rho_0$, and we assume that it is a probability distribution, $\rho_0\in \cP(\Omega)$. 

Equation~\eqref{e:agg-diff} arises as the infinite particle limit from systems of weakly interacting diffusions, first analysed by McKean~\cite{McKean1966,McKean1967}. It is rigorously shown in~\cite{Oelschlager1984,Sznitman1991,Stevens2000} that the law of the empirical distribution of the particle system converges to the solution $\rho$ of~\eqref{e:agg-diff} in the so-called mean field limit.

The derivation already shows that~\eqref{e:agg-diff}, maybe with an additional external potential, arises in many applications in which interactions between particles or agents are present. The examples range from opinion dynamics ~\cite{HegselmannKrause2002}, granular materials ~\cite{BCCP98,CarrilloMcCannVillani2003,BGG2013} and mathematical biology ~\cite{KellerSegel1971,Stevens2000,BCM2007} to statistical mechanics ~\cite{martzel2001mean}, galactic dynamics~\cite{binney2008}, liquid-vapor transitions~\cite{LebowitzPenrose1966,conlon2017}, plasma physics~\cite{bittencourt1986fund}, and synchronization ~\cite{kuramoto1981rhythms}.

In these models, all the physical properties of the underlying system are encoded through the interaction potential $W: \R^d\to \R$, which describes attractive and repulsive forces. We analyze the common case, where the force coming from the interaction potential between particles is antisymmetric and hence $W$ is assumed to be symmetric. In addition, we assume $W$ to be globally Lipschitz continuous and continuously differentiable away from the origin. These assumptions imply existence and uniqueness of weak solutions to~\eqref{e:agg-diff} by standard arguments (see for instance~\cite[Theorem 2.2]{CarrilloGvalaniPavliotisSchlichting2019}), but rule out singular potentials blowing up at the origin, like the Newtonian potential for the Keller--Segel model~\cite{KellerSegel1971}. In particular, $W$ is continuous on a bounded domain and thus finite, and we are allowed to add a constant, if necessary, in order to assure nonnegativity without effecting \eqref{e:agg-diff}. In summary, for later reference, our assumptions on the interaction potential are
\begin{align}
\label{A0} \tag{A1} &W(x) = W(-x) \text{ and } W(x) \geq 0 \:; \\
\label{A2}\tag{A2} &W\in C^1(\R^d \setminus\set{0}) \:;\\
\label{A3}\tag{A3} &W \text{ is Lipschitz-continuous.}
\end{align}

The aggregation-diffusion equation~\eqref{e:agg-diff} has a free energy functional acting as Lyapunov function for the evolution. It is defined as the sum of entropy and interaction energy,
\begin{equation}\label{e:free_energy}
 \cF(\rho) = \kappa \int_{\Omega} \rho(x) \log \rho(x) \dx{x} + \frac{1}{2} \iint_{\Omega\times \Omega} W(x-y) \rho(x) \rho(y)\dx{x} \dx{y} \:.
\end{equation}
We note that thanks to the nonnegativity assumption in \eqref{A0} and since the entropy of a probability measure is nonnegative, we have that $\cF(\rho)\geq 0$ on $\cP(\Omega)$.
A short computation reveals that the free energy dissipation is given by 
\begin{align}\label{e:FED}
  \pderiv{}{t} \cF(\rho) &= - \int_{\Omega} \rho \abs*{ \nabla \bra*{ \kappa \log \rho + W \ast \rho}}^2 \dx{x} =-\cD(\rho) .
\end{align}
It is natural to assume that the free energy is initially finite $\F(\rho_0)\in [0,\infty)$. 
Therewith, we can rewrite the aggregation-diffusion equation \eqref{e:agg-diff} as
\begin{equation}\label{e:agg-diff:GF}
  \partial_t \rho = \nabla\cdot \bra*{ \rho \bra*{ \kappa \nabla \log \rho + \nabla W \ast \rho}} = \nabla \cdot\bra*{ \rho \nabla \bra*{\frac{\delta\cF}{\delta \rho}(\rho)}} \:,
\end{equation}
where $\delta\cF/\delta\rho$ is the variational derivative (sometimes also called $L^2$ derivative) of $\cF$, and it becomes evident that~\eqref{e:agg-diff} is a gradient flow of the free energy $\cF$ with respect to the Wasserstein distance (cf.~\cite{JordanKinderlehrerOtto1998,CarrilloMcCannVillani2003,Villani2003}). In the thermodynamics interpretation, the term $\nabla(\delta\cF(\rho)/\delta\rho)$ is the (generalized) force driving the system towards states of lower free energy following the second law of thermodynamics.  Hence, in accordance with the principles of thermodynamics, the dissipation in \eqref{e:FED} can be expressed as the product of the force $\nabla(\delta\cF(\rho)/\delta\rho)$ and the flux $j=\kappa \nabla \rho + \rho \nabla W \ast \rho$ from~\eqref{e:agg-diff},
\begin{equation} \label{e:FED:force-flux}
 \cD(\rho) =  \int_{\Omega} \bra*{ \kappa\nabla\log \rho + \nabla W\ast \rho} \cdot \bra*{ \kappa \nabla\rho + \rho \nabla W\ast \rho}  \dx{x}.
  \end{equation}
  
The thermodynamical structure of the equation also provides a physical formulation of the stationary solutions to~\eqref{e:agg-diff}. On the one hand, it was already observed in~\cite{KirkwoodMonroe1941} that $\rho\in \cP(\Omega)$ is a stationary state if and only if $\rho$ is a fixed point of the Kirkwood--Monroe fixed point map $T:\cP_{\ac}(\Omega)\to \cP_{\ac}^+(\Omega)$ given by
\begin{equation}\label{e:KirkwoodMonroe}
  T(\rho)(x) = \frac{e^{-\kappa^{-1} W\ast \rho(x)}}{\int e^{-\kappa^{-1} W\ast \rho(y)}\dx{y}} ,
\end{equation}
where $\cP_{\ac}^{(+)}(\Omega)$ is the set of absolutely continuous positive probability measures.
The map $T$ allows to rewrite~\eqref{e:agg-diff} in symmetric state-dependent form
\begin{equation}\label{e:agg-diff:sym}
  \partial_t \rho = \kappa \nabla\cdot\bra[\Big]{ T(\rho) \nabla \frac{\rho}{T(\rho)}}.
\end{equation}
Hence, every stationary state has a Boltzmann statistical representation and it immediately follows that any equilibrium satisfies the detailed balance condition: $j=-\kappa T(\rho)\nabla \frac{\rho}{T(\rho)} = -\kappa \nabla \rho - \rho \nabla W \ast \rho=0$. The gradient flow formulation of~\eqref{e:agg-diff} on the other hand provides that stationary points are critical points of the energy $\cF(\rho)$ and they are non-dissipative, $\cD(\rho)=0$.

Finally, the free energy dissipation principle~\eqref{e:FED} is the starting point for studying the large time behavior of~\eqref{e:agg-diff}. After proving suitable lower-semicontinuity properties of $\cF$ and $\cD$, one can argue by the LaSalle invariance principle (see, e.g.,~\cite[Theorem 4.2 in Chapter IV]{Walker1980} and the recent preprint~\cite{CarrilloGvalaniWu2020}) that for $t\to \infty$ any solution converges to the set of stationary points of~\eqref{e:agg-diff}. This is proven for related models with specific interaction potentials in~\cite{Tamura1987,BCCP98,Tugaut2013,conlon2017}.

\subsection{Goals}

The main goal of this work is to provide a numerical scheme for a large class of aggregation-diffusion equations that preserve the free energy structure of the equation and is stable for all ranges of $\kappa$. In particular, the scheme should satisfy a discrete analog to the dissipation identity~\eqref{e:FED} and preserve the Boltzmann statistical form of stationary states in~\eqref{e:KirkwoodMonroe}.

For this purpose, we propose a (semi-)implicit finite volume scheme based on a Scharfetter--Gummel discretization of the flux. The particular form of the flux is physically meaningful as it can be derived from a suitable cell problem between neighboring control volumes (see~Section~\ref{sec:cell_problem}). Indeed, in~Proposition~\ref{thm:FED} below, we will recover the free energy dissipation principle for the Scharfetter--Gummel discretization in the product form of force times flux, cf.~\eqref{e:FED:force-flux}.

The free energy dissipation principle also leads our way towards a characterization of the stationary states: First, they are fixed points of a suitable discrete Kirkwood--Monroe map; second, they are critical points of the numerical free energy; and third, they are states with vanishing numerical dissipation. This observation is formulated in Theorem \ref{thm:stat} below. The findings underline our thesis that the proposed discretization is consistent with the thermodynamic structure of~\eqref{e:agg-diff}.

Finally, the numerical free energy dissipation identity gives rise to discrete regularity estimates (cf. Lemma~\ref{prop:grad_est}), which allow us to obtain compactness (cf.~Proposition~\ref{prop:compactness}) in order to conclude that the discretized solution converges to the weak solution of~\eqref{e:agg-diff} in Theorem~\ref{thm:longtime}. In addition, by exploiting the characterization of stationary states as critical points of $\cF$, we can also conclude that every stationary solution of the Scharfetter--Gummel scheme converges to a stationary solution of the continuous problem. 

The paper is organized as follows: In Section \ref{S.scheme}, we introduce the numerical Scharfetter--Gummel finite volume scheme approximating the aggregation-diffusion equation \eqref{1} and motivate the particular form of the flux by discussing the associated one-dimensional cell problem. We subsequently present and discuss our main results concerning well-posedness and convergence of the scheme, characterization of stationary states, and the large time behavior. The section concludes with a discussion of related work. The proofs are all contained in Section \ref{s:proofs}.

\section{The numerical scheme}\label{S.scheme}

This section will introduce the numerical scheme and derive some of its most elementary and substantial features. 

\subsection{Definition}\label{Ss.def_scheme}

We first introduce the general notation that is required to define a finite volume method. We start with the tesselation of the domain $\Omega$, which we assumed to be smoothly bounded earlier.   For technical reasons that we will briefly discuss later, it is convenient to choose a tesselation consisting of Voronoi cells. 
Hence, we let $\cT^h$ be a Voronoi tesselation covering $\Omega$ such that each element $K\in \cT^h$ is compact, has a non-trivial intersection with the physical domain, $K\cap \Omega\neq\emptyset$, and has maximal size $h$,
\begin{equation}\label{e:diam:h}
	\sup_{K\in \cT^h}\diam(K) \leq h \:.
\end{equation}
We set $\hat \Omega = \cup_K K$, which contains $\Omega$ by construction. The generator of each cell $K\in \cT^h$ will be denoted by $x_K$, and we set $d_{KL} = d(x_K,x_L)$, where $d$ is the Euclidean distance on $\R^d$.
If $K$ and $L$ are two neighboring cells, we write $L\sim K$, and we denote by $K\edge L$ the shared edge, $K\edge L = \bar K\cap \bar L$. We furthermore denote by $\tau_{KL}$ the transmission coefficients, given by 
$
\tau_{KL} = \frac{\abs{K\edge L}}{d_{KL}}$.
Here and in the following, the symbol $|\cdot|$ is used to denote an area, but we will also utilize it for volumes. Hence, $|K|$ is the volume of a cell $K$, and $|\partial K|$ is the area of its surface.

We finally discretize time. The time step size will be denoted by $\dt$ and we set $t^n = n\dt$ for any $n\in\N_0=\set{0,1,2,\dots}$.

With these preparations at hand, we are now in the position to introduce the finite volume approximations. 
To start with, we discretize the initial datum $\rho_0\in \cP(\Omega)$. 
By extending $\rho_0$ by zero to $\hat \Omega$, we may consider the averages $\rho_K^0 = \abs{K}^{-1}\int_K \rho_0\dx{x}$ on each cell $K\in \cT^h$, and set
$\rho^0_h = \{\rho_K^0\}_{K\in \cT^h}$. It is readily checked that the finite volume approximation $\rho_h^0$ of the initial configuration is a discrete probability distribution on $\cT^h$. The set of all probability distributions on $\cT^h$ will be denoted by $\cP(\cT^h)$, i.e.,
\[
\cP(\cT^h) = \set[\bigg]{ \set{\rho_K}_{K\in \cT^h}: \rho_K\geq 0 \ \forall K \in \cT^h \text{ and } \sum_{K} |K| \rho_K = 1}.
\]
The general iteration scheme that constitutes  a discrete evolution equation reads
\begin{equation}\label{e:time:flux}
	\abs{K} \frac{\rho^{n+1}_K - \rho^n_K}{\dt} + \sum_{L \sim K} F_{KL}^{n+1} = 0 ,
\end{equation}
where $F_{KL}^{n+1}$ is the numerical flux from cell $K$ to its neighbor $L$. Notice that there is no flux across the outer boundary $\partial\hat \Omega$ in accordance with \eqref{bc}. 
The Scharfetter--Gummel scheme approximates  the simultaneous flux due to diffusion \emph{and} advection across a common edge in terms of the Bernoulli function $B_\kappa:\R\to \R$ given for $\kappa>0$ by
\begin{equation}\label{e:def:B}
	B_\kappa(s) = \begin{cases}
		\frac{s}{e^{\frac{s}{\kappa}}-1} , & \text{ for } s\ne 0 , \\
		\kappa , & \text{ for } s= 0 . 
	\end{cases}
\end{equation}
This function   is convex, strictly decreasing, and satisfies
\[
B_\kappa(s) \geq (s)^- = \max\set{-s,0} \qquad\text{and}\qquad  \lim_{\kappa\to 0} B_\kappa(s) \to (s)^- \:,
\]
for any $s\in\R$. 
We propose the Scharfetter--Gummel numerical flux  approximation for the aggregation-diffusion equation \eqref{e:agg-diff} in the form 
\begin{equation} \label{e:def:F}
	F_{KL}^{n+1} = \tau_{KL} \bra*{ B_\kappa\bra*{d_{LK} q_{LK}^{n+1}} \rho^{n+1}_K - B_\kappa\bra*{d_{KL} q_{KL}^{n+1}} \rho^{n+1}_L} ,
\end{equation}
where $q_{KL}^{n+1}$ is a discretization of the aggregation convolution term,
\begin{equation} \label{e:def:q}
	q_{KL}^{n+1} = \sum_{J\in \cT} \abs{J}\  \frac{\rho_J^{n+1}+\rho_J^{n}}{2} \ \frac{W(x_K - x_J) - W(x_L - x_J)}{d_{KL}} \:. 
\end{equation}  
We will motivate this definition of the numerical flux briefly in  Subsection~\ref{sec:cell_problem}. The arithmetic mean occurring in time in~\eqref{e:def:q} is needed to have a numerical analogue of the free energy dissipation relation~\eqref{e:FED} for general interaction potentials $W$, see Theorem~\ref{thm:FED} below. The same choice was made in~\cite{BailoCarrilloHu2018,AlmeidaBubbaPerthamePouchol2019} to ensure energy dissipation for the upwind scheme. For further reference, we remark that both flux and convolution term are antisymmetric,
\begin{equation}\label{e:flux:antisymmetry}
	q_{KL}^{n+1} = - q_{LK}^{n+1} \qquad\text{and}\qquad F_{KL}^{n+1} = - F_{LK}^{n+1} .
\end{equation}

We conclude this subsection by stating three equivalent formulations of \eqref{e:def:F} and \eqref{e:def:q}, which will conveniently simplify later discussions and computations.  First, by introducing the unidirectional numerical fluxes from cell $K$ to $L$ denoted by $j_{KL}^{n+1}$, the flux can be written in divergence form,
\begin{equation}\label{e:def:uniflux}
	j_{KL}^{n+1} = B_\kappa\bra*{d_{LK} q_{LK}^{n+1}} \rho^{n+1}_K \qquad\text{and hence}\qquad F_{KL}^{n+1} = \tau_{KL} \bra*{ j_{KL}^{n+1} - j_{LK}^{n+1}} \:.
\end{equation}
Using the antisymmetry of the convolution term \eqref{e:def:q}, and the definition of the Bernoulli function \eqref{e:def:B}, we can furthermore write
\begin{equation}\label{e:flux:identity}
	\frac{F_{KL}^{n+1}}{\tau_{KL}} = j_{KL}^{n+1} - j_{LK}^{n+1} = d_{KL}q_{KL}^{n+1} \frac{\rho_K^{n+1} e^{\frac{d_{KL} q_{KL}^{n+1}}{2\kappa}}- \rho_L^{n+1} e^{-\frac{d_{KL} q_{KL}^{n+1}}{2\kappa}}}{ e^{\frac{d_{KL} q_{KL}^{n+1}}{2\kappa}} -e^{-\frac{d_{KL} q_{KL}^{n+1}}{2\kappa}}} \:.
\end{equation}
In fact, in our motivation for the particular form of the Scharfetter--Gummel flux, we will derive this identity rather than \eqref{e:def:F}. Finally, thanks to the elementary identity $B_{\kappa}(s) = \frac{s}2\bigl(\coth\bigl(\frac{s}{2\kappa}\bigr)-1\bigr)$, we may rewrite \eqref{e:def:F} as
\begin{equation}
	\label{e:flux:identity2}
	F_{KL}^{n+1} = |K\edge L| q_{KL}^{n+1} \frac{\rho_K^{n+1}+\rho_L^{n+1}}2 +\frac12 |K\edge L| q_{KL}^{n+1} \coth\Bigl(\frac{d_{KL} q_{KL}^{n+1}}{2\kappa}\Bigr)\left(\rho_K^{n+1}-\rho_L^{n+1}\right).
\end{equation}
This formulation of the numerical flux is beneficial, as it (roughly) separates the aggregation term from the diffusion term.

\subsection{Cell problem}\label{sec:cell_problem}

We briefly give some background information about the flux relation~\eqref{e:flux:identity} which is obtained from the solution of the following one-dimensional cell problem.

We consider two neighboring cells $K$ and $L$ with respective masses $\rho_K$ and $ \rho_L$ and intercellular advective flux $q_{KL}$.
For given $\rho_K, \rho_L$ and $q_{KL}$, the  normalized (per unit interfacial area) net flux $f_{KL}$ from cell $K$ to its neighbor $L$ is obtained as the solution of the boundary value problem
\begin{equation}\label{e:cellP_SG}
	\begin{aligned}
		f_{KL} &= -\kappa \partial_{x} \rho(\cdot) + q_{KL} \rho(\cdot) \quad \text{ on } (0, d_{KL}),\\
		\rho(0)&= \rho_K \qquad\text{and}\qquad  \rho(d_{KL})= \rho_L .
\end{aligned}\end{equation}
Hence, besides $f_{KL}\in \R$, the function $\rho:[0,1]\to \R$ is part of the unknown in~\eqref{e:cellP_SG}. It is readily checked that the Scharfetter--Gummel flux~\eqref{e:flux:identity} is computed as the solution to this two-point boundary problem  via the relation $ F_{KL} = \tau_{KL}f_{KL}$. This construction is the main reason why the Scharfetter--Gummel scheme preserves many of the structural properties of the continuous  equation~\eqref{e:agg-diff}.

It is straightforward to generalize this approach to aggregation-diffusion equations with diffusion operators having nonlinear mobilities, which is, for instance, the case for chemotaxis models avoiding overcrowding effects~\cite{BurgerDiFrancescoDolakStruss2006}. Another possible generalization is towards free energies of non-Boltzmann type like it is the case for the porous medium equation~\cite{Otto2001}. Although the resulting flux cannot be expressed as a simple closed function of $\rho_K, \rho_L$, and $v$, it still has many physical properties of the continuum equation. These generalizations are studied in~\cite{EymardFuhrmannGaertner06} for nonlinear diffusions with linear drifts. Similar constructions for generalized Scharfetter--Gummel schemes are known from the literature on numerical methods for semiconductors~\cite[Section 4.2]{Farrell_etal2017}. Moreover, various extensions of the Scharfetter-Gummel scheme for a problem with nonlocal (repulsive) interaction are discussed in~\cite{Cances_etal2019}. Further \emph{modified Scharfetter--Gummel} schemes are introduced in~\cite{BessemoulinChatard12} in the context of nonlinear diffusion equations with linear drifts to avoid working with non-explicit functions for the flux term. The generalization of these approaches to interaction equations with nonlinear diffusions remains to be investigated.

It might be insightful to study the solution $f_{KL}$ of the cell problem \eqref{e:cellP_SG} in terms of its parameters $\rho_K$, $\rho_L$, $d_{KL}$, $q_{KL}$, and $\kappa$. For this purpose, we introduce the solution function 
$\theta_\kappa:\R_+ \times \R_+ \times \R \to \R$ given by 
\begin{equation}\label{e:def:Theta}
	\theta_\kappa(a,b;v) = \begin{cases}
		\displaystyle v \, \frac{a e^{\frac{v}{2\kappa}} - b e^{-\frac{v}{2\kappa}}}{e^{\frac{v}{2\kappa}}- e^{-\frac{v}{2\kappa}}} , & \text{ for } v\ne 0 ,\\
		\kappa (a-b) , & \text{ for } v = 0,
	\end{cases}
\end{equation}
and write
\begin{equation}\label{e:flux:theta}
	f_{KL} = \theta_\kappa\bra*{\rho_K, \rho_L; d_{KL} q_{KL}}.
\end{equation}
To conclude this subsection, we gather some properties of this function that relate the Scharfetter--Gummel flux to physically known quantities and relations on the one hand and the numerical upwind scheme on the other hand.
\begin{remark}\label{lem:theta}
 For any $\kappa,a,b>0$, the following holds:
 \begin{enumerate}[ (i) ]
  \item The function $\theta_{\kappa}$ is \emph{one-homogeneous} in the first two variables, i.e.,  for any $\lambda>0$ and $v\in \R$, it holds that $\theta_\kappa(\lambda a, \lambda b; v) = \lambda \theta_\kappa(a,b;v)$.
  \item For small velocities, one recovers \emph{Fick's law} and the next order, i.e., 
  \begin{equation}\label{e:theta:small_v}
   \theta_\kappa(a,b;v)= \kappa(a-b) + \frac{a+b}{2} v + O(|v|^2)  \qquad\text{as } |v|\to 0 . 
  \end{equation}
  \item\label{rem3} The \emph{no-flux velocity} $v$ satisfying $\theta_\kappa(a,b;v)=0$ is characterized by  $v = -\kappa \log\frac{a}{b}$.
  \item It holds that $\R \ni v \mapsto \theta_\kappa(a,b;v)$ is one-to-one, since 
  \begin{equation}
\label{2}
  \min\set{a,b} \leq \partial_{v} \theta_\kappa(a,b;v) \leq \max\set{a,b}.
\end{equation}
  \item In the zero-diffusivity limit, $\theta_\kappa$ reduces to the \emph{upwind flux}, 
  \begin{equation}\label{e:theta:upwind}
  \lim_{\kappa \to 0} \theta_\kappa(a,b;v) = 
  \begin{cases}
    a v &, \text{ for } v>0, \\
    b v &, \text{ for } v<0 ,\\
    0 &, \text{ for } v = 0 .
  \end{cases} 
\end{equation}
 \end{enumerate}
\end{remark}

\subsection{Main results}
For our analysis, we have to ensure that the scheme is not degenerating in the sense that cells have the uniform isoperimetric property
\begin{equation}
\label{3}
\frac{|\partial K|}{|K|} \le \frac{C_{\iso}}{h}.
\end{equation} 
In order to show stability of the scheme in the sense of \eqref{e:stable}, we furthermore have to impose a smallness condition on both the mesh size and the time step size. A possible choice is
\begin{equation}\label{smallness_condition}
h \le \frac{\kappa}{\Lip(W)},\quad 64 C_{\iso} 
\frac{\Lip(W)^2 \cF^h(\rho^0)}{\kappa^2} \dt +512 C_{\iso}^2 
 \frac{\Lip(W)^4}{\kappa^2}(\dt)^2 \le 1,\quad \dt \frac{\Lip(W)^2}{\kappa} \le \frac{3}{8}. 
\end{equation}
Apparently, the conditions on $h$ and $\dt$ get more restrictive for small diffusivities  and irregular interaction potentials. However, we believe that in practice, any of these conditions are automatically satisfied. Moreover, we do not claim that any of the conditions in~\eqref{smallness_condition} is sharp. For instance, the precise assumption on the mesh size is chosen only to simplify the analysis.
\begin{theorem}[Well-posedness]\label{thm:well-posed}
(i) Under the assumption \eqref{3}, the Scharfetter--Gummel scheme \eqref{e:time:flux}, \eqref{e:def:F}, \eqref{e:def:q} with initial condition in $\cP(\cT^h)$ has at least one solution. Any solution is mass preserving and after the first time step strictly positive.
(ii) If $h$ and $\dt$ are in addition small in the sence of \eqref{smallness_condition}, the scheme is stable in $\ell^1(\cT^h)$. More precisely, there exists a constant $C$ such that for  any two probability distributions $\rho^0$ and $\tilde\rho^0 $ it holds that 
\begin{equation}\label{e:stable}
  \norm{\rho^n - \tilde\rho^n}_{\ell^1(\cT^h)} \leq C^n \norm{\rho^0 - \tilde\rho^0}_{\ell^1(\cT^h)} \:
\end{equation}
for any $n\in \N$. In particular, the solution is unique.
\end{theorem}
In analogy to the continuous setting~\eqref{e:free_energy}, we define the numerical free energy, which we split into its entropic part and interaction energy given by
\begin{align}
  \cF^h(\rho) &= \kappa \cS^h(\rho) + \cE^h(\rho), \label{e:free_energy:numerical}\\
  \text{with}\qquad \cS^h(\rho)&= \sum_K \abs{K} \rho_K \log \rho_K \label{e:entropy:numerical} \\
  \text{and}\qquad \cE^h(\rho) &=\frac{1}{2}  \sum_{K,L} \abs{K} \abs{L} W(x_K- x_L) \rho_K \rho_L \label{e:energy:numerical}\:.
\end{align}
We establish a discrete version of the free energy dissipation relation~\eqref{e:FED}.
Besides the free energy functional, we need to introduce the relative entropy between $\rho,\tilde\rho\in \cP(\cT^h)$ given by 
\begin{equation}\label{e:def:RelEnt}
 \cH(\rho\mid\tilde \rho) = \sum_{K} \abs{K} \rho_K \log \frac{\rho_K}{\tilde \rho_K} ,
\end{equation}
which will occur as an additional dissipation in time due to the implicit time discretization. 
\begin{proposition}[Free energy dissipation and large time behavior]\label{thm:FED}
  Let $\set{\rho_K^{n}}_{K,n}$ be a solution to the Scharfetter--Gummel scheme~\eqref{e:time:flux}, \eqref{e:def:F}, \eqref{e:def:q}. Then, for any $n\in\N_0$ it holds that
  \begin{equation}\label{e:FED:numerical}
   \frac{\cF^h(\rho^{n+1})-\cF^h(\rho^n)}{\dt} + \kappa \frac{\cH(\rho^n \mid \rho^{n+1})}{\dt}  = - \cD^h(\rho^{n+1})\:,
  \end{equation}  
  where $\cD^h$ is  the dissipation functional given by
  \begin{equation}\label{e:def:dissipation}
   \cD^h(\rho^{n+1}) = \sum_{K} \sum_{L\sim K} \frac{\kappa \abs{K\edge L}}{d_{KL}} \; \alpha_\kappa\bra*{\rho_K^{n+1},\rho_L^{n+1}, d_{KL} q_{KL}^{n+1}} \geq 0 ,
  \end{equation}
  where $\alpha_\kappa: \R_+ \times \R_+ \times \R \to \R_+$ is given in terms of $\theta_\kappa$ from~\eqref{e:def:Theta} by
\begin{equation}\label{e:def:alpha}
\alpha_{\kappa}(a,b;v) = \bra*{ \log \bra[\big]{a e^{\frac{v}{2\kappa}}} - \log\bra[\big]{b e^{-\frac{v}{2\kappa}}}} \, \theta_{\kappa}(a,b;v).
\end{equation}
\end{proposition}
We remark that due to the nonnegativity of  dissipation functional and relative entropy, the free energy functional is decreasing during the evolution,
\[
\cF^h(\rho^n) \le \cF^h(\rho^0),
\]
for any $n\in\N$. 
\begin{remark}\label{rem:SG:GF}
In comparison to~\eqref{e:FED}, we have the additional term $\cH(\rho^n \mid \rho^{n+1})$ providing a weak BV control on the discrete time gradient (cf.~Lemma~\ref{prop:grad_est}), which is typical for implicit schemes~\cite{SchlichtingSeis18}. 
To make the connection to~\eqref{e:FED} more apparent, we expand $\cD^h$ using~\eqref{e:theta:small_v} 
\begin{align*}
 \cD^h(\rho) &=\! \sum_{K} \sum_{L\sim K} \frac{|K\edge L|}{d_{KL}} \bra[\big]{ \kappa \bra*{\log \rho_K - \log \rho_L} + d_{KL} q_{KL}}\, \theta_\kappa(a,b;d_{KL} q_{KL}) \\
  &=\! \sum_{K} \sum_{L\sim K} d_{KL} |K\edge L| \bra[\big]{ \kappa \nabla_{KL}\log \rho + q_{LK}}\bra*{ \!\kappa \nabla_{KL}\rho + \frac{\rho_K + \rho_L}{2} q_{LK} + O(d_{KL} |q_{KL}|^2)\!} ,
\end{align*}
where we write $\nabla_{KL} \rho = \frac{\rho_L - \rho_K}{d_{KL}}$ for the discrete gradient. By recalling the form of $q$ as convolution with $W$ in~\eqref{e:def:q}, we 
expect in the continuum limit that~\eqref{e:FED:numerical} becomes~\eqref{e:FED}.

The validity of the numerical free energy dissipation principle raises the question of whether the stability estimate~\eqref{e:stable} can be improved to one which does not degenerate in the limit $h,\dt\to 0$. The classical entropy method~\cite{Yau1991} entails that two solutions of the limit equation~\eqref{e:agg-diff} are exponentially stable in relative entropy and hence also in $L^1$. Proposition~\ref{thm:FED} shows that Scharfetter--Gummel discretization is structure-preserving, and we conjecture that the entropy method is also applicable in this case.

Finally, we note that the Scharfetter--Gummel flux in~\eqref{e:def:Theta} is very close to a gradient flow formulation for the free energy~\eqref{e:free_energy:numerical}. For observing this, we note that the flux already encodes the discrete gradient of the logarithm. Hence, the \emph{Onsager relation}, which translates generalized forces like the derivative of the free energy to the flux $f_{KL}$, is given  in terms of the function
\begin{equation}\label{e:phi_kappa}
	\phi_\kappa(a,b;\xi) = \theta_\kappa\bra*{a,b;\xi - \kappa\log\frac{a}{b}} = 2 \kappa\sinh\bra*{\frac{\xi}{2\kappa}} \frac{\log\frac{e^{\frac{\xi}{2\kappa}}}{a}-\log\frac{e^{-\frac{\xi}{2\kappa}}}{b}}{\frac{e^{\frac{\xi}{2\kappa}}}{a}-\frac{e^{-\frac{\xi}{2\kappa}}}{b}}.
\end{equation}
The property~\eqref{2} ensures, that the mapping $\xi \mapsto \phi_\kappa(a,b;\xi)$ is indeed one-to-one. In this way, up to the semi-implicit term $\frac{\rho^{n+1}+\rho^{n}}{2}$ in the convolution, the scheme becomes close to a generalized gradient structure in the spirit of~\cite{LieroMielkePeletierRenger2017,peletier2020jump}
\begin{equation}\label{e:SG:GF}
	|K| \frac{\rho_K^{n+1}-\rho_K^n}{\delta t} = \sum_{L\sim K} \tau_{KL} \phi_\kappa\bra*{\rho_K, \rho_L; - d_{KL} \nabla_{KL} \frac{\delta \mathcal{F}^h}{\delta \rho}(\rho^n)} .
\end{equation}
In particular the upwind scheme from~\cite{BailoCarrilloHu2018} is obtained by replacing $\phi_\kappa$ in~\eqref{e:SG:GF} with its limit for $\kappa\to 0$ given by $\phi_0(a,b;\xi) = a\, \xi_+ - b \, \xi_-$. 
\end{remark}
The stationary solutions to~\eqref{e:time:flux} are densities $\rho\in \cP(\cT^h)$ such that for all $K\in \cT^h$ it holds
\begin{equation}\label{e:stat:def}
0 = \sum_{L\sim K}   F_{KL}[\rho]=  \sum_{L\sim K} \tau_{KL}\bra*{ j_{KL}[\rho] - j_{LK}[\rho]} = \sum_{L\sim K} \tau_{KL} \theta_\kappa\bra*{ \rho_K,\rho_L ; d_{KL} q_{KL}[\rho]} , 
\end{equation}
where we used the identities~\eqref{e:def:uniflux} and~\eqref{e:flux:theta} with $\rho^{n+1}=\rho^{n}=\rho$. 
The stationary states have the following characterization, which is completely analogous to the situation for the continuous aggregation-diffusion equation studied in~\cite[Proposition 2.4]{CarrilloGvalaniPavliotisSchlichting2019}.
\begin{theorem}[Characterization of stationary states]\label{thm:stat}
Let $h>0$ and  $\rho\in \cP(\cT^h)$ be given. The following statements are equivalent:
\begin{enumerate}[ (i) ]
	\item\label{thm:stat:1} $\rho$ is a stationary state satisfying~\eqref{e:stat:def};
	\item\label{thm:stat:2} $\rho$ solves the Kirkwood--Monroe fixed point equation
	\begin{equation}\label{e:stat:fixed_point}
		\rho_K  = \frac{\exp\bra*{-\kappa^{-1} \sum_{J} \abs{J} \rho_J W(x_K-x_J)}}{Z^h(\rho)} , 
	\end{equation}
	for any $K\in \cT^h$, where $Z^h(\rho)= \sum_{\tilde K} \abs{\tilde K} \exp\bra*{-\kappa^{-1} \sum_{J} \abs{J} \rho_J W(x_{\tilde K}-x_J)}$;
	\item\label{thm:stat:3} $\rho$ is a critical point of the free energy functional $\cF^h$ on $\cP(\cT^h)$;
	\item\label{thm:stat:4} $\rho$ is without dissipation, i.e., $\cD^h(\rho)=0$.
\end{enumerate}
\end{theorem}
Note that the fixed point identity entails, in particular, that any stationary solution $\rho$ is strictly positive.

With this characterization of stationary states and thanks to the free energy dissipation relation~\eqref{e:FED:numerical} we can establish the large time behavior of the scheme. More precisely, we show that for any initial datum, approximate solutions converge towards a stationary state. As a particular consequence, this result proves that the set of stationary states has to be non-empty. Since our argument relies on the stability estimate in \eqref{e:stable}, we have to impose that smallness condition on $h$ and $\dt$.
\begin{theorem}[Large time behavior of the scheme]\label{thm:longtime}
Suppose that \eqref{smallness_condition} holds.  Let $\set{\rho_K^{n}}_{K,n}$ be a solution to the Schar\-fetter--Gummel scheme~\eqref{e:time:flux}, \eqref{e:def:F}, \eqref{e:def:q}. Let $\varPi^h$ be the set of stationary solutions. Then $\rho^n$ approaches $\varPi^h$ in the large time limit,
\begin{equation}\label{e:longtime}
	\lim_{n\to \infty} \dist_{\ell^1(\cT^h)}(\rho^n, \varPi^h) = 0 . 
\end{equation}
\end{theorem}
The statement~\eqref{e:longtime} is equivalent to the existence of a stationary state $\overline\rho^h\in \varPi^h$ and of a subsequence $n_k$ such that $\rho^{n_k} \to \bar\rho^h$ in $\ell^1(\cT^h)$.  A sufficient condition for the uniqueness of the limit would be that $\varPi^h$ consists only of isolated stationary points. Once this is established, we expect that for specific choices of $W$, some quantified convergence rates to stationary states could be established translating, for instance, the method of~\cite{CarrilloMcCannVillani2003} to the discrete setting. We leave a deeper understanding of the dependency of the set of stationary solutions $\varPi^h$ on the interaction potential $W$ for future work.

Given the discrete solution $\{\rho_K^n\}$ of a numerical finite volume scheme, we can introduce the \emph{approximate solution}  $\rho_{\dt,h}\in L^1(\R_+\times \hat \Omega)$ of the PDE, given by
\[
\rho_{\dt, h} = \rho_K^n\quad \text{for a.e. }(t,x)\in \left[t^n,t^{n+1}\right) \times K.
\]
The estimate~\eqref{e:FED:numerical} also carries a priori gradient information about the discrete gradient in the dissipation functional. Exploiting this, we can establish sufficient compactness of the discrete scheme in the limit $h\to 0$ to obtain a  convergence result.
\begin{theorem}[Convergence]\label{thm:convergence}
	The approximate solution sequence $\{\rho_{\dt,h}\}$ converges as $h\to0$ and $\dt\to0$ in $L^2((0,T);L^1(\Omega))$ towards a distributional  solution of the continuous problem. Moreover, any stationary solution of the Scharfetter--Gummel scheme converges in $L^1(\Omega)$ to a stationary solution of the continuous problem.
\end{theorem}

\subsection{Numerical experiments}

We implement the scheme writing the flux as in ~\eqref{e:flux:identity} and using the solution function $\theta_\kappa$ of the cell problem defined in~\eqref{e:def:Theta}. Hence, we consider the flux functional
\begin{equation}\label{e:flux:numerical}
	\frac{F_{KL}^{n+1}[\rho^{n+1}]}{\tau_{KL}} = \theta_\kappa\bra*{\rho_K^{n+1},\rho_L^{n+1}, d_{KL} q_{KL}^{n+1}} 
\end{equation}
with interaction fluxes $q_{KL}^{n+1}$ defined in~\eqref{e:def:q},
and rewrite the scheme~\eqref{e:time:flux} as the  fixed point equation
\begin{equation}\label{e:numeric:fixpoint}
	\rho^{n+1}_K = \rho^{n}_K + \frac{\delta t}{\abs{K}} \sum_{L \sim K} F_{KL}^{n+1}[\rho^{n+1}] . 
\end{equation}
For comparison, we also study the upwind discretization of the flux term proposed in~\cite{BailoCarrilloHu2018} and given by
\begin{equation}\label{e:flux:numerical:upwind}
	\frac{\tilde F_{KL}^{n+1}[\rho^{n+1}]}{\tau_{KL}} = \rho_K^{n+1}\bra*{Q_{KL}^{n+1}}_+ - \rho_L^{n+1}\bra*{Q_{KL}^{n+1}}_- \quad\text{with}\quad Q_{KL}^{n+1} = \kappa \log \frac{\rho_K^{n+1}}{\rho_L^{n+1}} + d_{KL} q_{KL}^{n+1} .
\end{equation}
Hereby, we truncate the logarithm at the machine precision for avoiding possible occurrences of \texttt{NaN}. By this construction, the upwind scheme has the same stationary states characterized by the fixed point mapping~\eqref{e:stat:fixed_point}, which allows us to compare the long-time behavior of both methods.

In addition to the evolution equation, we also implement the fixed point mapping~\eqref{e:stat:fixed_point} to detect stationary states besides the uniform state. By comparing their free energies, we can judge which of those is the global minimum. The free energy gap to this global minimum is used in the following plots on a semilog scale.

Our implementation is written in the Julia language~\cite{Julia-2017}, and the fixed point problem~\eqref{e:numeric:fixpoint} is solved using a Newton method through the NLSolve-package~\cite{NLSolve}. 
In our numerical test, we use an adaptive time-stepping algorithm to resolve metastable situations and to take advantage of the implicit time. More specifically, the time-step $\dt$ is initially chosen to be~$h$ and later adjusted to be larger if the number of Newton steps is small and smaller if the number of  Newton steps is large.\footnote{The exact strategy is that the time step is halved if more than $4$ Newton iterations are needed and increased by $10\%$ if less than $2$ Newton iterations are performed.} The simulation is stopped if the free energy gap is of order $10^{-15}$ or the change of free energy during several subsequent time-steps is of the order of the machine precision.

In the Figure~\ref{Fig:Kuramoto:subsup},
\begin{figure}[t]
	\centering
	\includegraphics[width=0.48\textwidth]{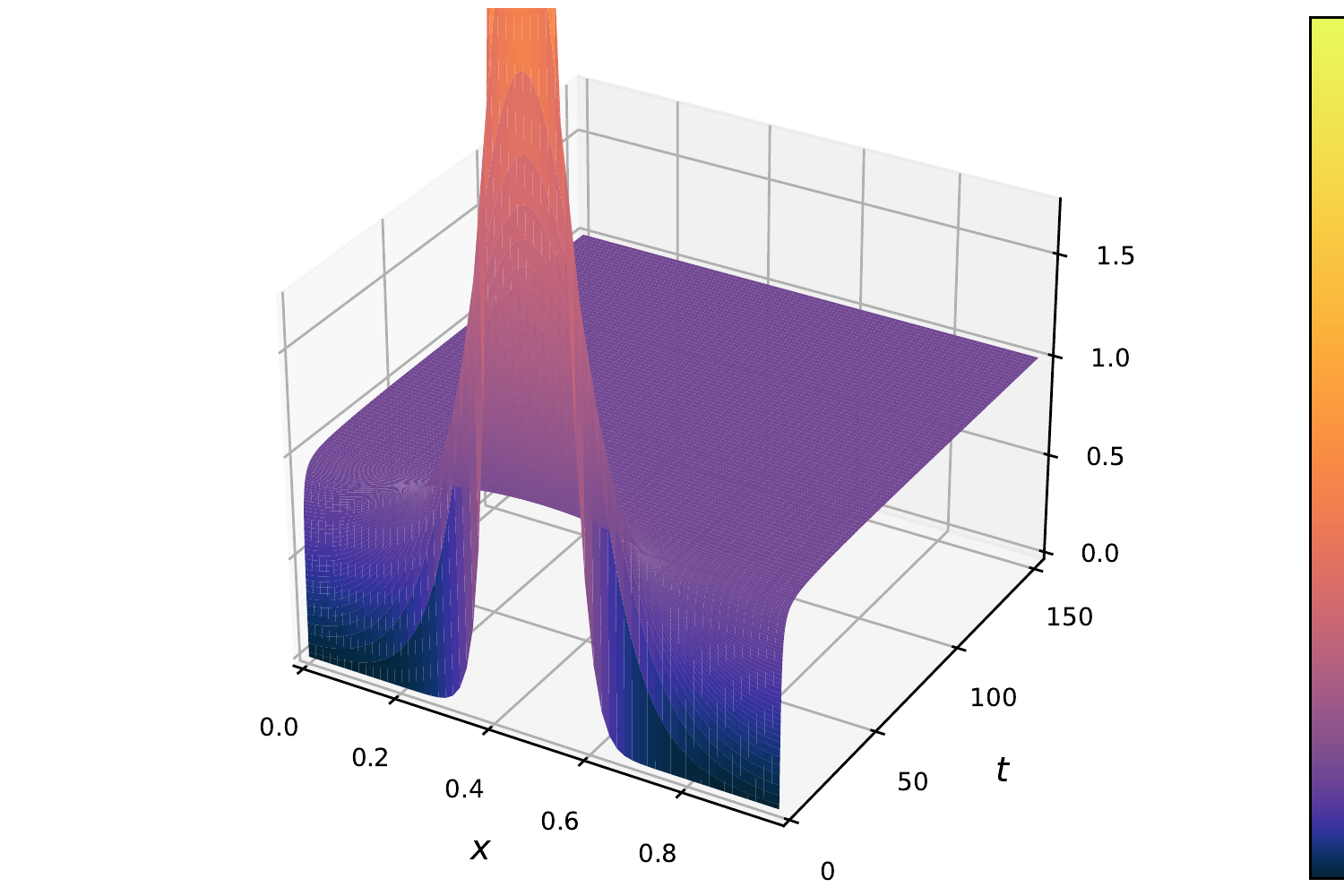}
	\hfill
	\includegraphics[width=0.48\textwidth]{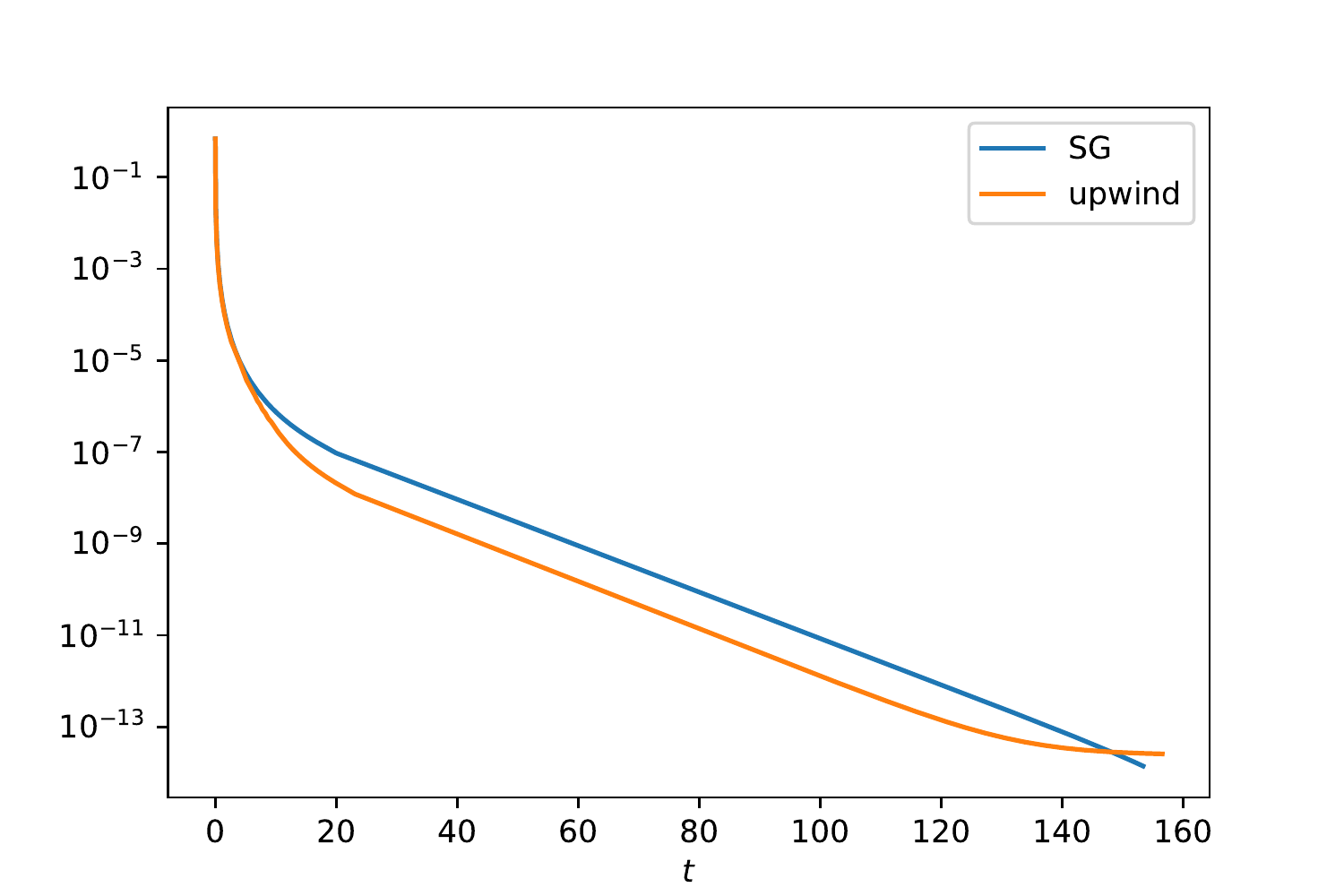}\\
	
	\includegraphics[width=0.48\textwidth]{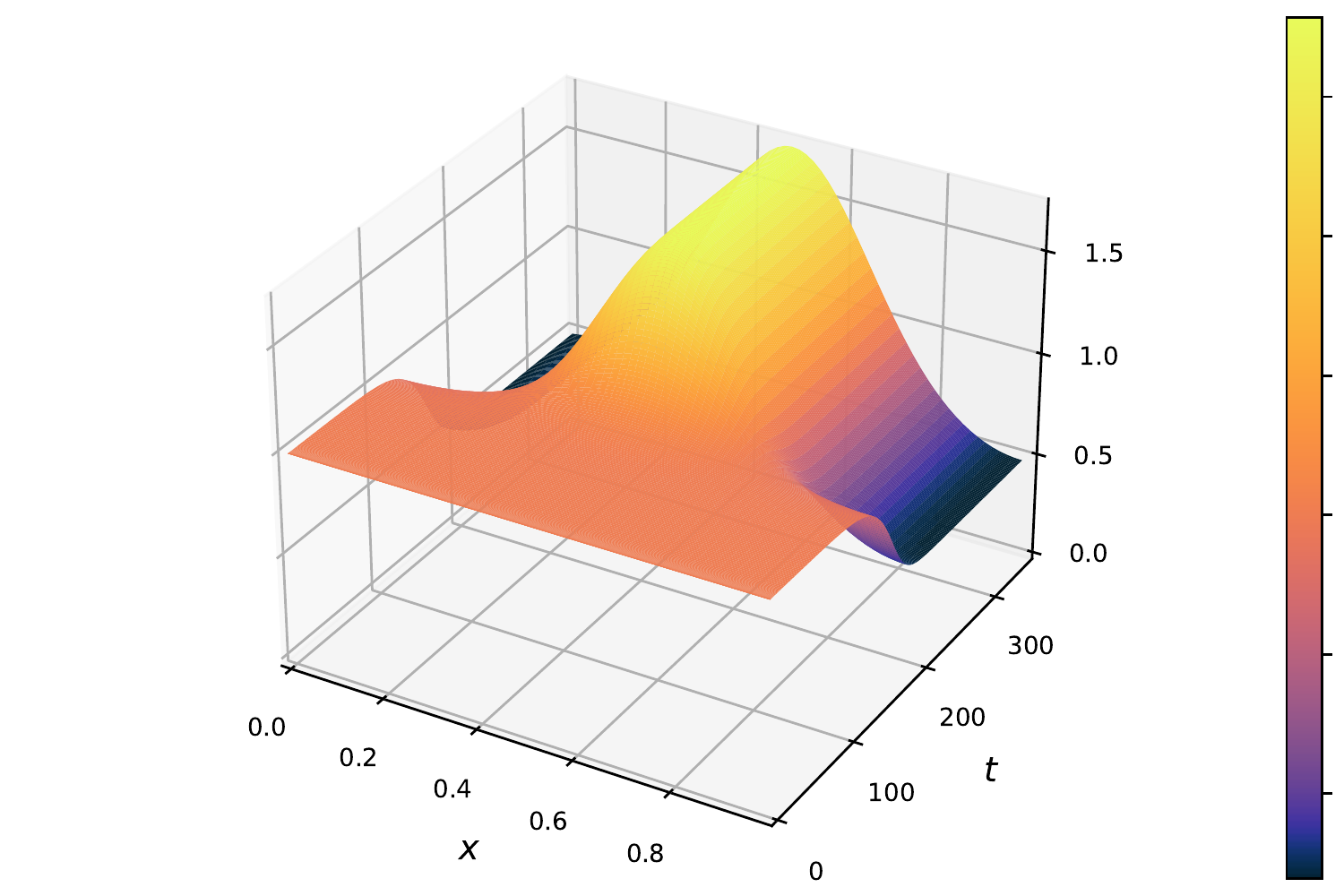}
	\hfill
	\includegraphics[width=0.48\textwidth]{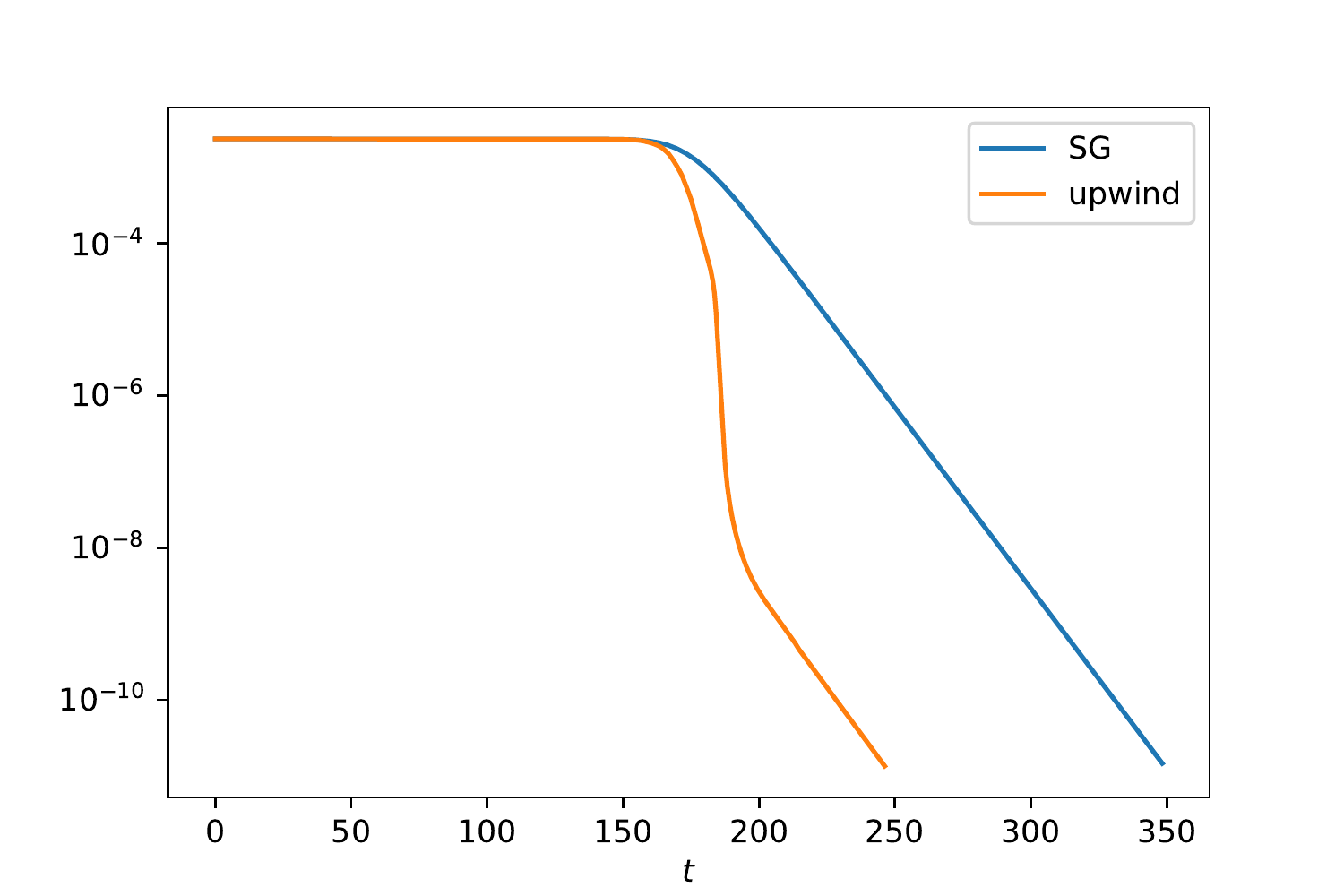}
	\caption{The evolution for the Kuramoto model of $\rho^h$ (right) and the free energy in semi-log-scale compared to the upwind scheme. The parameters are $\kappa=1$,
		where $\sigma= 1.9$ (first row) for the subcritical regime and $\sigma=2.1$ (second row) for the supercritical regime.
		We work on a uniform mesh on the torus $[0,1]$ discretized with $h=2^{-6}$.}\label{Fig:Kuramoto:subsup}
\end{figure} 
we show the evolution of the classical Kuramoto model, in which the interaction principle is given by $W(x) = -\sigma \cos(2\pi x) $ for some positive real $\sigma$. A continuous phase transition occurs in this model for $\sigma=2$, and we solve the evolution equation for sub- and supercritical values. We observe that both schemes can resolve this model and converge in free energy to the unique stable stationary state. Hereby, the convergence of the upwind scheme is faster in physical time, and we believe that this effect is due to excess numerical diffusion introduced by the scheme. For this model, it is expected that the free energy relaxes exponentially fast provided the solution is close enough to a local minimum of the free energy. We can observe this behavior for the proposed Scharfetter--Gummel scheme, whereas the upwind scheme shows a slightly artificial behavior in this regard. We also stress that despite the seemingly faster relaxation computed by the upwind scheme in physical time, the computation costs are comparable (see Table~\ref{Tab:ComputCost}.
\begin{table}[t]
	\begin{tabular}{|ll|l|l|l|l|}
		\hline
		Figure & scheme & runtime [s] & steps & alloc $\times 1e6$ & mem  [GiB] \\ \hline
		Fig.~\ref{Fig:Kuramoto:subsup}: $\sigma=1.9$& SG & $39$ & $171$ & $1.0$ & $16$ \\ \hline
		Fig.~\ref{Fig:Kuramoto:subsup}: $\sigma=1.9$& upwind & $36$ & $189$ & $1.1$ & $16$ \\ \hline
		Fig.~\ref{Fig:Kuramoto:subsup}: $\sigma=2.1$& SG & $53$ & $279$ & $1.5$ & $23$ \\ \hline
		Fig.~\ref{Fig:Kuramoto:subsup}: $\sigma=2.1$& upwind & $65$ & $287$ & $2.0$ & $30$\\ \hline
		Fig.~\ref{Fig:Gauss} & SG & $210$ & $116$ & $5.5$ & $83$ \\ \hline
		Fig.~\ref{Fig:Gauss}& upwind & $540$ & $279$ & $17$ & $250$ \\ \hline
	\end{tabular}\\[6pt]
	\caption{\label{Tab:ComputCost} Run-time, computational costs in terms of allocations and total memory usage measured with the Julia \texttt{@time} command. All simulations are run on an AMD Ryzen 5 PRO 3500U with 16GiB memory.}
\end{table}

Other test cases are systems with a discontinuous phase transition, where dynamical metastability is expected. This is generically the case for attractive interaction potentials that are sufficiently localized. We refer to~\cite[Lemma 5.3]{GvalaniSchlichting20} for the specific assumptions on $W$ to ensure this situation.
In Figure~\ref{Fig:Gauss}, 
\begin{figure}[t]
	\includegraphics[width=0.48\textwidth]{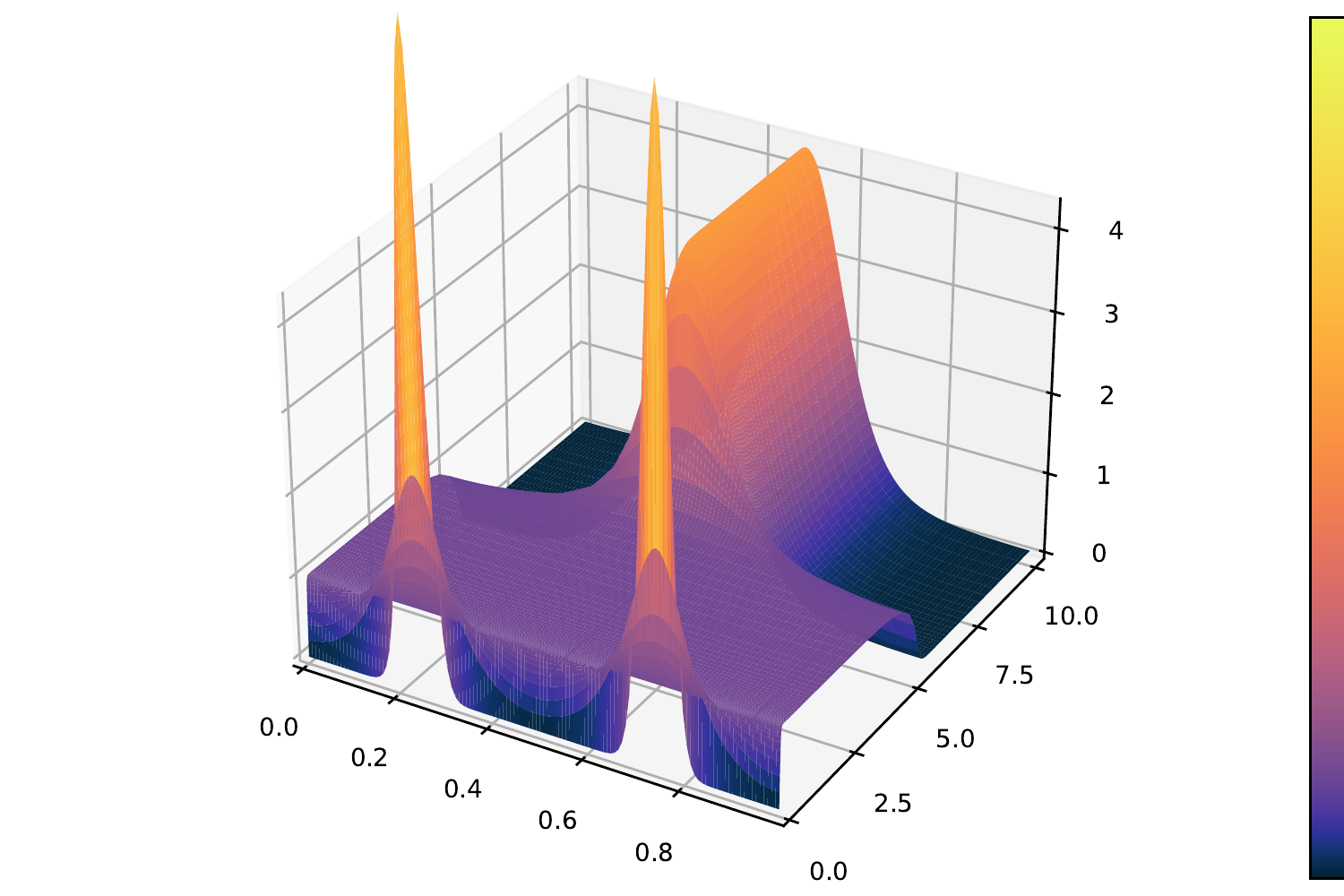}
	\hfill
	\includegraphics[width=0.48\textwidth]{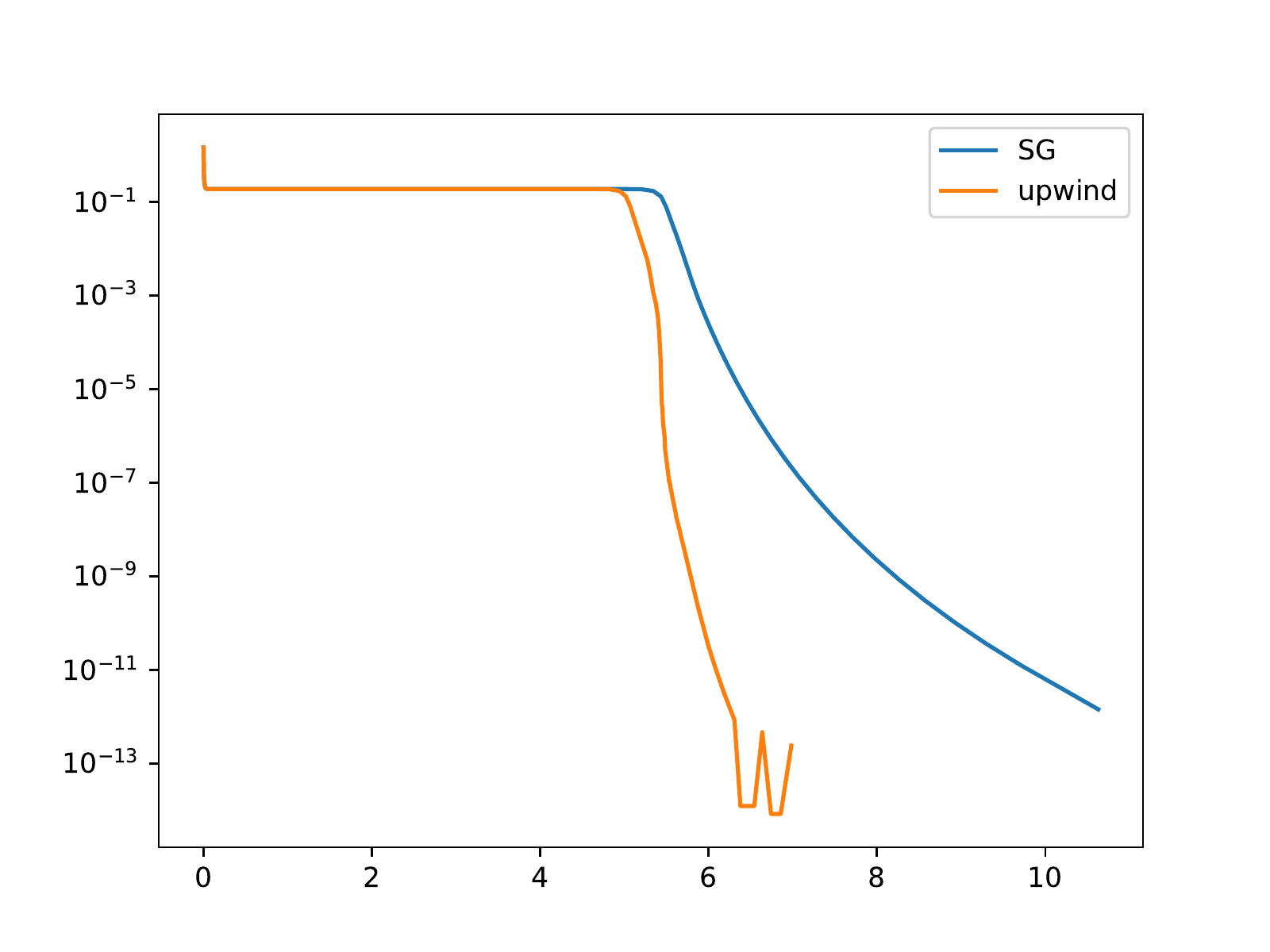}
	\caption{\label{Fig:Gauss} The evolution of $\rho^h$ (right) and the free energy in semi-log-scale in comparison to the Upwind scheme from~\cite{BailoCarrilloHu2018}. The parameters are $h=2^{-7}$, $\kappa=1$.}
\end{figure}
an interaction potential of the form of a negative Gaussian with a sufficiently small variance is chosen, $
W(x) = -6 \exp(-\frac{\abs{x}^2}{\sigma})$ with $ \sigma = 0.05 $, which leads to the formation of a discontinuous phase transition as proven in~\cite{CarrilloGvalaniPavliotisSchlichting2019}. 
The simulation is started from an initial state with two concentrated peaks at $0.25$ and $0.75$ and an additional perturbation at $0.5$, which ensures the convergence of both schemes towards the global minimum featuring a single peak at $0.5$. Let us emphasize that this model needs a structure-preserving resolution of the flux since the distance of the solution at time $0.5$ to the uniform state in $L^\infty$ is of order $10^{-8}$.
The free energy for the upwind scheme shows some artificial convergence behavior due to probably increased numerical diffusivity. 
We again emphasize that although the convergence in physical time is faster for the Upwind scheme, the numerical cost is a factor of $2.5$ larger in comparison to the Scharfetter-Gummel scheme as reported in Table~\ref{Tab:ComputCost}. The adaptive time-stepping for upwind performs worse because, on average more Newton steps are performed for the convergence to the fixed point of~\eqref{e:flux:numerical:upwind}.

\subsection{Related work}

A considerable amount of numerical schemes has been proposed and analyzed for advection-diffusion equations in general and aggregation-diffusion equations in particular. In the following, we give a (necessarily incomplete) list of references. 

The class of schemes that we adapt in the present work goes back to Scharfetter and Gummel in 1969~\cite{SG1969}. In this work, the authors' objective is to derive effective numerical methods for simulating semiconductors, where both advection and diffusion-dominated regimes can occur. Further contributions for numerically solving semiconductor models include, for instance, mixed exponential methods, e.g.,~\cite{BrezziMariniPietra89,Juengel95,JuengelPietra1997}, or upwind finite volume schemes, e.g.,~\cite{ChainaisHillairetLiuPeng03,ChainaisHillairetPeng03,ChainaisHillairetPeng04}. For more references about the developments of schemes for drift-diffusion models, with thorough background on semiconductor models, see, e.g.,~\cite{Farrell_etal2017,Cances_etal2019}.

In the context of aggregation-diffusion equations~\eqref{e:agg-diff}, there has been quite some research activity in recent years, with a special focus on structure-preserving properties. For instance, in \cite{BailoCarrilloHu2018,Bailo_etal2020}, a (semi-)implicit discretization of~\eqref{e:agg-diff} and nonlinear variants based on upwind fluxes is proposed and shown to converge. In~\cite{Liu2019}, the starting point for a scheme is formula~\eqref{e:agg-diff:sym}, which is discretized by symmetric differences. This work shows positivity and free energy dissipation. Still, the scheme in~\cite{Liu2019} has the drawback that it is formulated for $\kappa=1$ only, and we expect stability issues for such an approach for small diffusivity constants.
In \cite{AlmeidaBubbaPerthamePouchol2019}, two schemes for Fokker--Planck and generalized Keller--Segel models are proposed. The approach, which the authors also dedicate to Scharfetter and Gummel, is based on the symmetric reformulation~\eqref{e:agg-diff:sym}, which is then discretized together with some upwind technique. Hence, it differs from our more problem-specific choice of the flux interpolation~\eqref{e:flux:identity}, but is explicit also for nonlinear mobilities.
We also mention Lagrangian particle approximations, e.g.,~\cite{CarrilloCraigPatacchini2019}, which are motivated by  Lagrangian particle interpretations of aggregation equations \cite{CDFLS11}.  For further references on this subject, we refer to the recent review~\cite{CarrilloCraigYao2019}.

In the present work, we consider linear diffusions, but also a generalization of the Scharfetter--Gummel scheme to aggregation equations with nonlinear diffusions seems possible. Related to this, we mention \cite{EymardFuhrmannGaertner06,BessemoulinChatard12}, where advection-diffusion equations with nonlinear diffusions (and nonlinear fluxes) are studied. Notice that the work \cite{BessemoulinChatard12} is similar to ours in the sense that steady states and large time dynamics are investigated. (Here, the motivation is again a model for semiconductors.)  Both works and the fact that many relevant nonlinear aggregation-diffusion equations own a free energy functional similar to~\eqref{e:free_energy} suggest that analysis of the finite free energy setting seems to be accessible also for nonlinear diffusions.

We want to mention that the theoretical investigation of the large time behavior of Scharfetter--Gummel-related schemes for advection-diffusion equation is contained in~\cite{BessemoulinChatardChainaisHillairet2017,Cances_etal2019-large,LiLiu2020} based on discrete function inequalities, that go back to~\cite{chatard_asymptotic_2011,BessemoulinChatard2014}. Already before, the scheme in \cite{ChainaisHillairetFibet07} uses a nonlinear approximation of the diffusion flux, which in the linear setting would essentially rely on the observation that the Laplacian can be interpreted as an advection operator with nonlinear flux according to the identity $\laplace \rho = \div(\rho\grad\log\rho)$, and shows preservation of the large time behavior of the continuous model. 

Concerning the rate of convergence for this and related schemes, we refer to the recent works~\cite{LagoutiereVauchelet2017,DelarueLagoutiereVauchelet2017}, which show for the upwind scheme applied to the pure aggregation equation, i.e., $\kappa=0$, an explicit rate of order $\tfrac{1}{2}$ with respect to the Wasserstein distance. Since the proposed Scharfetter--Gummel scheme is a direct generalization to the case with a diffusion of strength $\kappa$, we conjecture that a similar analysis is possible in this case. We leave this project for future research.

To make direct use of the gradient flow formulation~\eqref{e:agg-diff:GF}, it seems natural to construct a discretization based on an implicit-in-time minimizing movement scheme as introduced in~\cite{JordanKinderlehrerOtto1998}. Although analytically appealing, such an approach usually suffers from the high cost for solving an implicit Euler scheme involving the Wasserstein distance in each time step. 
Recent advances that circumvent this problem are based on the Benamou--Brenier dynamical formulation of the Wasserstein distance~\cite{BB2000}. For instance, in~\cite{BenamouCarlierLaborde2016}, the computation of the Wasserstein distance is turned in each time step into a saddle point problem, whereas in~\cite{LiLuWang2020} the authors convexify the minimizing movement scheme via a suitable entropic regularization. 

In the present work, we instead followed the \emph{discretize first, optimize later} paradigm, which brought us to the Scharfetter--Gummel interpolation through the cell-problem in Section~\ref{sec:cell_problem}.
Concerning the free energy dissipation relation~\eqref{e:FED:numerical} it is a natural question whether the Scharfetter--Gummel scheme inhabits a gradient structure already on the discrete level. 
Indeed, as pointed out already in Remark~\ref{rem:SG:GF} about the Onsager relation, our scheme can be considered as a generalized gradient flow in the spirit of~\cite{LieroMielkePeletierRenger2017,peletier2020jump}.
Here, the notion \emph{generalized} refers to the Onsager relation~\eqref{e:phi_kappa} being a monotone but nonlinear function of the force. In comparison, previous works on gradient flow formulations of Markov chains~\cite{CHLZ12}, motivated by the upwind scheme and also the discrete McKean--Vlasov dynamics on graphs~\cite{EFLS2016}, a linear relationship is used. This is also true for~\cite{DisserLiero2015,Heida_etal2020}, whose discretizations of Fokker--Planck equations seem to be limited to linear drift terms.

Related to this discussion, we finally mention the two works~\cite{CancesGallouetTodeschi2019,EPSS19}, which provide a gradient flow formulation of the upwind scheme on finite volumes. In fact, the variational description in those approaches reduces the Onsager relation~\eqref{e:phi_kappa} to the upwind flux, that is $\phi_\kappa(a,b,\xi)\to \phi_0(a,b,\xi)= a\, \xi_+ - b \, \xi_-$ for $\kappa \to 0$.

\section{Proofs}\label{s:proofs}

\subsection{Existence and first properties}
In this subsection we provide the proof of Theorem~\ref{thm:well-posed} except for the stability estimate, which will be provided in Subsection \ref{SS:free_energy} below.

The following two lemmas represent general results for the Scharfetter--Gummel scheme and do not rely on the specific choice of the driving vector field $q_{KL}$ in~\eqref{e:def:q}. 
First, we show that the scheme is conservative. 
\begin{lemma}\label{L1} 
Let $\rho_{\dt, h}$ be an approximate solution to the Scharfetter--Gummel scheme~\eqref{e:time:flux} and~\eqref{e:def:F} with antisymmetric driving vectorfield $\set{q_{KL}^n}_{K,L,n}$, that is $q_{KL}^n=-q_{KL}^n$ for all $K,L\in \cT^h$ and $n\in \N_0$. Then $\rho_{\dt, h}$ is mass preserving in the sense that
\begin{equation}\label{1}
\int_{\Omega} \rho_{\delta t, h}(t,x)\dx{x} = \int_{\Omega} \rho_h^0(x)\dx{x},
\end{equation}
for any $t>0$.
\end{lemma}
\begin{proof}
The conservativity \eqref{1} of the scheme is a consequence of the antisymmetry~\eqref{e:flux:antisymmetry}, which implies by symmetrization
\[
 \sum_{K} \sum_{L\sim K} F_{KL}^{n+1} = \frac{1}{2}\sum_{K} \sum_{L\sim K} \bra*{F^{n+1}_{KL} + F^{n+1}_{LK}} = 0 
\]
and hence
\begin{equation}\label{e:scheme:conserve}
 \sum_{K} \abs{K} \rho^{n+1}_K  = \sum_{K} \abs{K} \rho^{n}_K + \dt \sum_{K} \sum_{L\sim K} F^{n+1}_{KL} = \sum_{K} \abs{K} \rho^{n}_K .
\end{equation}
The statement in \eqref{1} now follows by iteration and the fact that $\int_{\Omega} \rho_{\dt, h}\dx{x} =\sum_K|K|\rho_K^n$ for some $n\in\N_0$.
\end{proof}

Our next goal is to show the positivity of the scheme.
\begin{lemma}\label{L2}
Let $\rho_{\dt, h}$ be an approximate solution to the Scharfetter--Gummel scheme~\eqref{e:time:flux} and~\eqref{e:def:F} with antisymmetric driving vectorfield $\set{q_{KL}^n}_{K,L,n}$ of potential form, that is for some $\set{V_K^n}_{K,n}$ it holds
\[
  d_{KL} q_{KL}^n = V_K^n - V_L^n \qquad\text{for all } K,L\in \cT \text{ and } n\in \N_0 .
\]
If the initial datum $\rho^0_h$ is nonnegative, then $\rho_{K}^n$ is positive for all $K\in \cT^h$ and $n\in \N$.
\end{lemma}
\begin{proof}[Proof of Lemma~\ref{L2}]
 For the proof, it is convenient to consider the transformed quantity  $h_K^n = \rho_K^n e^{\kappa^{-1} V_K^n}$ instead of $\rho_K^n$. A short computation reveals that $h^n_K$ solves the iteration scheme
 \begin{align}
 |K| \frac{h_K^{n+1}e^{-\kappa^{-1} V_K^{n+1}} -h_K^n e^{-\kappa^{-1} V_K^n}}{\dt} &=\sum_{L\sim K} \tau_{KL} \frac{V_{K}^{n+1}-V_L^{n+1}}{e^{\frac1{\kappa}V_K^{n+1} }-e^{\frac1{\kappa}V_{L^{n+1}}}} \left(h_L^{n+1}-h_K^{n+1}\right) \nonumber \\
  & = \kappa \sum_{L\sim K} \tau_{KL} \frac{ h_L^{n+1}-h_K^{n+1}}{\Lambda\bra*{e^{\frac1{\kappa} V_K^{n+1}},e^{\frac1{\kappa}  V_L^{n+1}}}} ,\label{e:scheme:symmetric}
 \end{align}
where $\Lambda:\R_+\times \R_+\to \R_+$ is the logarithmic mean given by
  \begin{equation}
   \Lambda(a,b)=\begin{cases}
        \frac{a-b}{\log a- \log b} , & a,b> 0 ,\\
        a , & a=b   .                                                                                                                            
                 \end{cases}
  \end{equation}
 We observe that  \eqref{e:scheme:symmetric} is the implicit finite volume scheme for an elliptic linear operator  and hence satisfying the maximum principle, that is, if $h_K^{n}\geq 0$ for all $K\in \cT^h$, we immediately get $h_K^{n+1}\geq 0$ for all $K$. Indeed, suppose that for some $n$ and $J\in \cT^h$, it holds $h_{J}^{n+1}= \min_{K\in \cT^h} h_{K}^{n+1}<0$ and $h_{K}^n \geq 0$ for all $K\in \cT^h$, then we have by \eqref{e:scheme:symmetric},
 \begin{align*}
0>  |J| \frac{h_J^{n+1}e^{-\kappa^{-1} V_J^{n+1}} -h_J^n e^{-\kappa^{-1} V_J^n}}{\dt} =\kappa \sum_{L\sim J} \frac{\tau_{JL}}{\Lambda\bra*{e^{\frac1{\kappa} V_J^{n+1}},e^{\frac1{\kappa}  V_L^{n+1}}}}  \bra*{h_L^{n+1}-h_{J}^{n+1}}  \geq 0 ,
 \end{align*}
which is a contradiction.
 
To obtain the positivity, let $h_{L}^{n}\geq 0$ for all $L\in \cT^h$ and $h_{L^*}^{n}>0$ for some $L^*\in \cT^h$. Moreover, towards a contradiction, we assume that $h_{K^*}^{n+1}=0$ for some $K^*\in \cT^h$. By using $K=K^*$ in~\eqref{e:scheme:symmetric}, we immediately see that
 \[
 \frac{|K|}{\dt}h_{K^*}^{n} e^{-\frac{1}{\kappa} V_{K^*}^{n+1}} + \sum_{L\sim K^*} \tau_{KL} \frac{h_L^{n+1}}{\Lambda(e^{\frac1{\kappa}V_{K^*}^{n+1}},e^{\frac1\kappa V_L^{n+1}})} = 0.
 \]
 By the nonnegativity of $\{h_K^n\}_n$, this implies that $h^{n+1}_L=0$ for any $L\sim K^*$ because the denominators above are finite. This argument can be reapplied to any $K\sim K^*$, and  by iteration we arrive at $h_{K}^{n+1}=0$ for any $K\in \cT^h$. Inserting  this information back into \eqref{e:scheme:symmetric}, we deduce that $h_{K}^{n}=0$ for any $K\in \cT^h$, which contradicts the hypothesis that $h_{L^*}^n>0$.
\end{proof}
It remains to show that the scheme is well-defined. 
\begin{proof}[Proof of Theorem~\ref{thm:well-posed}(i)] Existence of the finite volume approximation is obtained via a standard topological degree argument. We provide the details for the convenience of the reader. We introduce the homotopy parameter $\mu\in[0,1]$ and consider the continuous mapping $\Xi = \{\Xi_K\}_{K}$ given by
\[
\Xi_K(\varphi,\mu)=
|K| \frac{\varphi_K - \rho^n_K}{\dt} + \mu \sum_{L\sim K} F_{KL}[\varphi,\rho^n] .
\]
We first observe that any solution $\varphi= \rho_{\mu}^{n+1}$  to the equation $\Xi[\varphi,\mu]=0$ is nonnegative and mass-preserving by Lemmas \ref{L1} and \ref{L2}, and thus
\begin{equation}\label{600}
0 \le \rho_{\mu,K}^{n+1} \le \frac{\|\rho^{n+1}_{\mu}\|_{L^1(\hat \Omega)}}{|K|} \le \frac{\|\rho^0\|_{L^1(\hat \Omega)}}{\min_{K\in\cT}|K|} =:m.
\end{equation}
In particular, $\rho_{\mu}^{n+1}\not\in \partial  M$, where $M = [-1, 1+m]^{\# \cT}$. It follows that the degree $\deg(\Xi(\cdot,\mu),M,0)$ is well-defined, and by homotopy invariance, it holds that
\[
\deg(\Xi(\cdot, \mu),M,0) = \deg(\Xi(\cdot ,0),M,0) = \deg(\Xi(\cdot ,1),M,0).
\]
Clearly, $\Xi(\cdot,0)$ is an affine map, which has degree $1$, and thus, in particular,  $\Xi(\cdot ,1)$ has a non-zero degree. Hence, $\Xi(\varphi ,1)=0$ has a solution $\varphi= \rho_{1}^{n+1}$, which satisfies \eqref{e:time:flux} by definition.
\end{proof}

\subsection{Free energy dissipation principle and consequences}\label{SS:free_energy}
\begin{proof}[Proof of Proposition~\ref{thm:FED}]
By the result of Lemma~\ref{L2}, we have that $\rho^{n+1}_K >0$ for all $K\in \cT^h$, hence all occurrences of $\log \rho^{n+1}_K$ are well-defined. We start with the entropy and use the convention that $0\cdot \log 0 = 0 = 0 \cdot \log \infty$ to arrive at
\begin{equation}\label{full_entropy_dissipation}
  \frac{\cS^h(\rho^{n+1})-\cS^h(\rho^n)}{\dt} = \sum_{K} \log\bra*{\rho_K^{n+1}}  \abs{K} \frac{\rho_K^{n+1} - \rho_K^n}{\dt} +\frac{1}{\dt} \sum_{K}\abs{K} \rho_K^n \log \frac{\rho_K^{n+1}}{\rho_K^n} \:.
\end{equation}
We proceed by recalling the definition of relative entropy from~\eqref{e:def:RelEnt} and  use the scheme~\eqref{e:time:flux} with identity \eqref{e:def:uniflux} and a symmetrization
\begin{align}
 \MoveEqLeft{\frac{\cS^h(\rho^{n+1})-\cS^h(\rho^n) + \cH(\rho^n\mid \rho^{n+1})}{\dt}= -\sum_{K}\sum_{L\sim K} \log \rho_K^{n+1} F_{KL}^{n+1}}\\
 &= -\frac{1}{2} \sum_{K} \sum_{L\sim K} \abs{K\edge L}\, \frac{\log \rho_K^{n+1} - \log \rho_L^{n+1}}{d_{KL}} \, \bra*{ j_{KL}^{n+1} - j_{LK}^{n+1}} .
\end{align}
Now, we turn to the discrete interaction energy~\eqref{e:energy:numerical}. Here, we use the following discrete product rule for $K,J\in \cT^h$
\begin{equation}
  \rho_K^{n+1} \rho_J^{n+1} - \rho_K^n \rho_J^n = \bra*{ \rho_K^{n+1} - \rho_K^n} \frac{\rho_J^{n+1} + \rho_J^n}{2} + \bra*{ \rho_J^{n+1} - \rho_J^n} \frac{\rho_K^{n+1} + \rho_K^n}{2} \:,
\end{equation}
and the symmetry of $W$ in \eqref{A0} to write
\[
\frac{\cE^h(\rho^{n+1})-\cE^h(\rho^n)}{\dt} = \sum_{K,J} W(x_J - x_K) \abs{J}\frac{\rho_J^{n+1}+\rho_J^n}{2} \abs{K} \frac{\rho_K^{n+1} - \rho_K^n}{\dt} .
\]
Applying the scheme \eqref{e:time:flux} with \eqref{e:def:uniflux} and another symmetrization, we arrive at
\begin{align}
  \MoveEqLeft{\frac{\cE^h(\rho^{n+1})-\cE^h(\rho^n)}{\dt}   }\\
 &= -\sum_{J,K} \sum_{L\sim K} W(x_J - x_K) \abs{J}\frac{\rho_J^{n+1}+\rho_J^n}{2} \frac{\abs{K\edge L}}{d_{KL}} \bra*{ j_{KL}^{n+1} - j_{LK}^{n+1}}\\
 &= -\frac{1}{2} \sum_{K} \sum_{L\sim K} \abs{K\edge L} \bra*{ j_{KL}^{n+1} - j_{LK}^{n+1}} \sum_J \abs{J}\frac{\rho_J^{n+1}+\rho_J^n}{2} \frac{W(x_J -x_K) - W(x_J -x_L)}{d_{KL}} \\
 &= - \frac{1}{2}\sum_{K} \sum_{L\sim K}  \abs{K\edge L}  \, q_{KL}^{n+1}  \, \bra*{ j_{KL}^{n+1} - j_{LK}^{n+1}} \:,
\end{align}
by definition~\eqref{e:def:q}.

Finally, we combine both calculations and invoke  \eqref{e:flux:identity} to prove~\eqref{e:FED:numerical},
\begin{align}
  \MoveEqLeft{\frac{\cF^h(\rho^{n+1})-\cF^h(\rho^n)+ \cH(\rho^n\mid \rho^{n+1})}{\dt}} \\
  &= - \frac{1}{2}\sum_{K} \sum_{L\sim K}  \abs{K\edge L}  \,\frac{\kappa \log \rho_K^{n+1} - \kappa \log \rho_L^{n+1} + d_{KL} q_{KL}^{n+1}}{d_{KL}}  \, \bra*{ j_{KL}^{n+1} - j_{LK}^{n+1}} \\
  &=- \frac{1}{2}\sum_{K} \sum_{L\sim K} \kappa \abs{K\edge L}  \, \frac{\log \bra*{ \rho_K^{n+1} e^{\frac{d_{KL} q_{KL}^{n+1}}{2\kappa}}}-\log\bra*{ \rho_L^{n+1} e^{-\frac{d_{KL} q_{KL}^{n+1}}{2\kappa}}}}{d_{KL}} \times \\
  &\qquad\qquad\qquad\qquad \times d_{KL}q_{KL}^{n+1} \frac{\rho_K^{n+1} e^{\frac{d_{KL} q_{KL}^{n+1}}{2\kappa}}-\rho_L^{n+1} e^{-\frac{d_{KL} q_{KL}^{n+1}}{2\kappa}}}{ e^{\frac{d_{KL} q_{KL}^{n+1}}{2\kappa}} -e^{-\frac{d_{KL} q_{KL}^{n+1}}{2\kappa}}} . 
\end{align}
This concludes the proof of Proposition \ref{thm:FED}.
\end{proof}
In the following, our goal is to derive gradient estimates from the free energy dissipation principle. For this, we assume the following smallness condition on the mesh size,
\begin{equation}
\label{e:h_small}
h \le \frac{\kappa}{\Lip(W)},
\end{equation}
which is also needed in order to establish stability, see \eqref{smallness_condition}.
\begin{lemma}[Gradient estimates]
\label{prop:grad_est}
The following estimates hold true for any $N\in \N$,
\begin{align}
\kappa^2 \dt \sum_{n=0}^N \sum_K \sum_{L\sim K}\frac{|K\edge L|}{d_{KL}} \Bigl(\sqrt{\rho_K^{n+1}}-\sqrt{\rho_L^{n+1}}\Bigr)^2  &\lesssim \F^h(\rho^0) + T\Lip(W)^2,\label{Fisher_information}\\
\label{spatial_gradient}
\kappa^2 \dt \sum_{n=0}^N   \biggl( \sum_K \sum_{L\sim K}|K\edge L|\,| \rho_K^{n+1}-\rho_L^{n+1}|\biggr)^2 &\lesssim \F^h(\rho^0) + T\Lip(W)^2,\\
 \label{time_derivative}
\kappa \sum_{n=0}^N \biggl(\sum_K |K| \, |\rho_K^{n+1}-\rho_K^n|\biggr)^2 & \le \F^h(\rho^0).
\end{align}
\end{lemma}
The reader who is familiar with the entropy theory for the heat equation will identify our first estimate \eqref{Fisher_information} as a discrete version of the entropy bound on the  Fisher information, which is a consequence of the dissipation estimate \eqref{e:FED}. On the  discrete level, these estimates rely on Proposition \ref{thm:FED}.

We furthermore remark that the gradient estimates \eqref{Fisher_information} and  \eqref{spatial_gradient}  make use of the smallness condition in \eqref{e:h_small}. We recall that this is not a condition for existence or positivity. Moreover, for very small diffusivities for which \eqref{e:h_small} is violated, a variation of the analysis below would yield the weak BV estimate
\[
 \sum_K\sum_{L\sim K}| K\edge L| \, |q^{n+1}_{KL}| \, |\rho_K^{n+1}-\rho_L^{n+1}| \le  \sum_K \sum_{L\sim K} |K\edge L| \bigl(\rho^{n+1}_K (q^{n+1}_{KL})^+ + \rho_L^{n+1}(q_{KL}^{n+1})^-\bigr)\lesssim 1,
\]
which is known for the upwind scheme \cite{SchlichtingSeis18}. We will make use of the gradient estimate in~\eqref{spatial_gradient} in our derivation of stability estimates. More precisely, we will use for all $n$ the bound
\begin{equation}
\label{fine_gradient_estimate}
\left(\sum_K \sum_{L\sim K} |K\edge L| |\rho_K^{n+1} - \rho_L^{n+1}|\right)^2 \le 5\frac{ C_{\iso} \cF^h(\rho^0)}{\dt \, \kappa^2} +  32 \frac{C_{\iso}^2 \Lip(W)^2}{\kappa^2} ,
\end{equation}
which can be established by carefully tracking the constants in the following proof, see in particular~\eqref{e:gradient:quantitative} and~\eqref{e:gradient:quantitaive2}. Of course, we do not claim that the bound above is sharp. Yet, it eventually gives an explicit estimate on the time step size that is needed in order to establish our stability estimate.
\begin{proof}
We start by noticing that, because $(\log(a)-\log(b))(a-b)(a+b) \ge 2(a-b)^2 $ for any positive $a,b>0$ by the concavity of the logarithm, we have 
\[
\alpha_{\kappa}(a,b,v) \ge \frac{(ae^{\frac{v}{2\kappa}} - be^{-\frac{v}{2\kappa}})^2}{a e^{\frac{v}{2\kappa}} + be^{- \frac{v}{2\kappa}}}\frac{v}{e^{\frac{v}{2\kappa}}-e^{-\frac{v}{2\kappa}}}.
\]
For notational simplicity, we write $z = \frac{v}{2\kappa}$ and suppose that $|z|\le 1$ in the following, which we will justify later. We furthermore set
\[
\phi(z) = \frac{(a e^z - be^{-z})^2}{a e^z + be^{-z}},
\]
and notice that by Taylor expansion and the elementary inequality $a\,b  \le \eps a^2 + \frac1{4\eps}b^2$, for any $a,b,\eps >0$, it holds
\begin{align*}
\phi(z) & \ge \phi(0) + \phi'(0)z -\frac12 \max_{y\in [-1,1]} |\phi''(y)| z^2\\
& \ge \frac{(a-b)^2}{a+b} +  2(a-b)z -\frac{(a-b)^3}{(a+b)^2}z - C_1(a+b)z^2\\
&\ge \frac12\frac{(a-b)^2}{a+b} - C_2(a+b) z^2,
\end{align*}
with constants $C_1=4e$ and $C_2=4e+5\leq 16$.

For $v = d_{KL}q_{KL}^{n+1}$ and our smallness assumption \eqref{e:h_small}, we indeed have that $|z|\le 1$. Truly, in view of the definition of $q^{n+1}_{KL}$ in \eqref{e:def:q} and because the approximate solution is a probability measure, it holds that $|z|  = {d_{KL}|q_{KL}^{n+1}|}/{2\kappa}  \le \frac{h}{\kappa}\Lip(W)\le 1$. 
Since $|e^z-e^{-z}|\leq (e-1/e) |z|$ for $|z|\le1$. In total, we obtain the lower bound
\[
 \kappa \alpha_{\kappa}(a,b,v) = \kappa\phi\bra*{\frac{v}{2\kappa}} \frac{2\kappa \frac{v}{2\kappa}}{e^{\frac{v}{2\kappa}}-e^{-\frac{v}{2\kappa}}} \geq \frac{\kappa^2}{e-1/e}\frac{(a-b)^2}{a+b} - \frac{8}{e-1/e} (a+b)v^2  . 
\]
We thus infer from Proposition \ref{thm:FED} and the nonnegativity of the relative entropy $\cH (\rho^{n+1},\rho^n)$, see the discussion below, that
\begin{align*}
\frac{\F^h(\rho^{n+1}) - \F^h(\rho^n)}{\dt} &+ \frac{\kappa^2}{e-1/e} \sum_K \sum_{L\sim K} \frac{ |K\edge L|}{d_{KL}} \frac{(\rho_K^{n+1} - \rho_L^{n+1})^2}{\rho_K^{n+1} + \rho_L^{n+1}}\\ &\leq  \frac{8}{e-1/e}\sum_K \sum_{L\sim K}\frac{|K\edge L|}{d_{KL}} (\rho_K^{n+1} + \rho_L^{n+1}) \left| d_{KL} q^{n+1}_{KL} \right|^2.
\end{align*}
Notice that the term on the right-hand side without prefactor is bounded by
\[
\Lip(W)^2 \sum_K \sum_{L\sim K} d_{KL} |K\edge L| (\rho_K^{n+1} + \rho_L^{n+1}) \leq 2 C_{\iso} \Lip(W)^2 \|\rho^{n+1}\|_{L^1(\Omega)}  = 2 C_{\iso} \Lip(W)^2,
\]
where we have used  relabeling arguments, the isoperimetric property \eqref{3}, and a similar reasoning as above in order to bound $q_{KL}^{n+1}$. Hence, summing over $n$ and multiplying by~$\dt$ and $(e-1/e)\leq 2.5$ gives 
\begin{equation}\label{e:gradient:quantitative}
\kappa^2 \dt \sum_K \sum_{L\sim K}\frac{|K\edge L|}{d_{KL}}\frac{(\rho_K^{n+1}-\rho_L^{n+1})^2}{\rho_K^{n+1} + \rho_L^{n+1}} \leq 2.5	\F^h(\rho_0) + 16 C_{\iso} \dt \Lip(W)^2.
\end{equation}
It remains to apply the elementary inequality $(\sqrt{a} - \sqrt{b})^2 \le \frac{(a-b)^2}{a+b}$ and to sum in $n$ to deduce \eqref{Fisher_information}.

In order to prove \eqref{spatial_gradient}, we have to use the Cauchy--Schwarz inequality,
\begin{align}
\MoveEqLeft \kappa^2 \dt \biggl(\sum_K\sum_{L\sim K} |K\edge L| \, |\rho_K^{n+1}-\rho_L^{n+1}|\biggr)^2 \nonumber \\
&  \le \kappa^2 \dt \biggl(\sum_K \sum_{L\sim K} \frac{|K\edge L|}{d_{KL}} \frac{(\rho_K^{n+1}-\rho_L^{n+1})^2}{\rho_K^{n+1}+\rho_L^{n+1}}\biggr) \biggl(\sum_K\sum_{L\sim K} |K\edge L|d_{KL} (\rho_K^{n+1}+\rho_L^{n+1})\biggr) \nonumber \\
&\leq 2 C_{\iso}  \kappa^2 \dt \biggl(\sum_K \sum_{L\sim K} \frac{|K\edge L|}{d_{KL}} \frac{(\rho_K^{n+1}-\rho_L^{n+1})^2}{\rho_K^{n+1}+\rho_L^{n+1}}\biggr) \label{e:gradient:quantitaive2}
\end{align}
where we notice that the second term is again bounded by $2 C_{\iso}$.
After a summation in~$n$, this proves \eqref{spatial_gradient}.

The estimate of the time derivative in \eqref{time_derivative} is an immediate consequence of the entropy dissipation \eqref{e:FED:numerical} in Proposition \ref{thm:FED}. Indeed,  dropping the  dissipation  term on the right-hand side of \eqref{e:FED:numerical}, summing over $n$ and using the  Pinsker inequality yields
\[
\frac{\kappa}2 \sum_{n=0}^N \biggl(\sum_K |K| \, | \rho^{n+1}_K - \rho_K^n|\biggr)^2 \le \kappa  \sum_{n=0}^N \cH(\rho^{n+1},\rho^n) \le \cF^h(\rho^0).
\]
This concludes the proof of Lemma \ref{prop:grad_est}.
\end{proof}
The previous lemma provides bounds on gradients in terms of the discrete free energy of the initial datum $\rho^0=\rho_h^0$. Our next lemma ensures that this quantity is indeed bounded for finite entropy initial data $\rho_0$. We will make use of this observation in the convergence analysis in the subsection that follows. Moreover, this bound makes sure that the right-hand side of \eqref{fine_gradient_estimate} is bounded uniformly in $h$ for $h$ sufficiently small.
\begin{lemma}[Initial energy]
It holds that
\begin{equation}
\label{ini_dat}
\F^h(\rho^0_h) \le \F(\rho_0) + h\Lip(W).
\end{equation}
\end{lemma}
\begin{proof}
From the convexity of the function $\varphi(z) =  z\log z$, we obtain via Jensen's inequality that
\[
\cS^h(\rho_h^0)  =\sum_K |K| \varphi\biggl(\frac1{|K|} \int_K \rho_0\dx{x}\biggr)  \le \sum_K \int_K \varphi(\rho_0)\dx{x} = \int_{\Omega} \varphi(\rho_0)\dx{x} =  \cS(\rho_0).
\]
Moreover, using the definitions of the interaction energy and the Lipschitz property of the potential $W$, 
\begin{align*}
\cE^h(\rho_h^0) &  =\frac12 \sum_{K,L} \int_K\int_L W(x_K-x_L) \rho_0(x)\rho_0(y)\dx{x}\dx{y}\\
&\le \frac12 \int_{\Omega}\int_{\Omega} W(x-y)\rho_0(x)\rho_0(y)\dx{x}\dx{y} + h\Lip(W) \sum_{K,L} \int_K\rho_0\dx{x}\int_L\rho_0\dx{y}\\
& = \cE(\rho_0) + h\Lip(W).
\end{align*}
A combination of both estimates yields the statement of the lemma.
\end{proof}
We are now in the position to show the stability estimate of Theorem \ref{thm:well-posed}.
\begin{proof}[Proof of Theorem \ref{thm:well-posed}(ii)]
We consider the function $\sigma(w) = w(\coth(w)+1)$ and observe that
\[
\theta_{\kappa}(a,b,v) = \kappa \, \sigma\bra*{\frac{v}{2\kappa}}a + \kappa\, \sigma\bra*{-\frac{v}{2\kappa}}b.
\]
Given positive numbers $a,\tilde a,b,\tilde b$ and reals $w=v/2\kappa$ and $\tilde w = \tilde v/2\kappa$, we have 
\begin{align*}
\kappa^{-1}\theta_{\kappa}(a,b,v) - \kappa^{-1}\theta_{\kappa}(\tilde a,\tilde b, \tilde v) & = \sigma(w)a -\sigma(\tilde w)\tilde a + \sigma(-w) b - \sigma(-\tilde w)\tilde b\\
&  = \sigma(w)(a-\tilde a)  +\sigma(-w)(b-\tilde b)\\
&\qquad +(\tilde a - \tilde b)(\sigma( w)-\sigma(\tilde w))\\
&\qquad +\tilde b \left((\sigma( w)+\sigma( -w)) - (\sigma(\tilde w)+\sigma(-\tilde w))\right).
\end{align*}
Now, if $\rho^{n+1}$ and $\tilde \rho^{n+1}$ are two solutions to the Scharfetter--Gummel scheme obtained from $\rho^n$ and $\tilde \rho^n$, we have according to this way of  splitting that
\begin{align*}
|K|\left(\rho^{n+1}_K-\tilde \rho^{n+1}_K\right)  & = |K| \left(\rho^n_K-\tilde \rho^n_K\right) - \dt \kappa \sum_{L\sim K} \tau_{KL} \sigma(w_{KL}) \left(\rho_K^{n+1}-\tilde \rho_K^{n+1}\right)\\
&\qquad  - \dt\kappa \sum_{L\sim K}\tau_{KL} \sigma(w_{LK})\left(\rho_L^{n+1}-\tilde \rho_L^{n+1}\right)\\
&\qquad +\dt \kappa \sum_{L\sim K} \tau_{KL} \left(\tilde \rho_K^{n+1} - \tilde \rho_L^{n+1}\right)\left(\sigma(\tilde w_{KL} )- \sigma(w_{KL})\right)\\
&\qquad +\dt \kappa \sum_{L\sim K} \tau_{KL} \tilde \rho_L^{n+1} \left((\sigma (\tilde w_{KL})+\sigma(\tilde w_{LK})) - (\sigma(w_{KL})+\sigma(w_{LK})\right).
\end{align*}
Here, we have used the notation $w_{KL} = \frac{d_{KL}}{2\kappa} q_{KL}[\rho^{n+1},\rho^n]$ and the symmetry $w_{KL}=-w_{LK}$. (Of course, $\tilde w_{KL}$ is analogously defined.)
For $w$ and $\tilde w$ in the interval $[-1,1]$, we have the elementary estimates
\begin{gather*}
|\sigma(w) - \sigma(\tilde w)| \le \frac85|w-\tilde w|,\\
|(\sigma(w) + \sigma(-w)) - (\sigma(\tilde w) + \sigma(-\tilde w))|  = 2| w\coth(w)-\tilde w\coth(\tilde w)| \le \frac23 |w+\tilde w||w-\tilde w|.
\end{gather*}
Hence, supposing without loss of generality that $\rho_K^{n+1}\ge \tilde \rho_K^{n+1}$ and using the fact that $\sigma $ is nonnegative, we get that
\begin{align*}
\MoveEqLeft
|K||\rho_K^{n+1} - \tilde \rho_K^{n+1}| +\dt \kappa \sum_{L\sim K} \tau_{KL} \sigma(w_{KL}) |\rho_K^{n+1}-\tilde \rho_K^{n+1}|\\
&\le |K| |\rho_K^n-\tilde \rho_K^n| + \dt\kappa \sum_{L\sim K} \tau_{KL} \sigma(w_{LK}) |\rho_L^{n+1}-\tilde \rho_L^{n+1}| \\
&\qquad +\frac{8}5 \dt \kappa \sum_{L\sim K} \tau_{KL} |\tilde \rho_K^{n+1} - \tilde \rho_L^{n+1}| |w_{KL}- \tilde  w_{KL}| \\
&\qquad + \frac23 \dt \kappa \sum_{L\sim K} \tau_{KL} \tilde \rho_L^{n+1}|w_{KL}+\tilde w_{KL}| |w_{KL}-\tilde w_{KL}|.
\end{align*}
Summing over $K$ and relabelling, we observe that the second terms on both sides of the inequality cancel out, and we arrive at
\begin{equation}
\label{702}
\|\rho^{n+1} - \tilde \rho^{n+1}\|_{L^1(\hat \Omega)}  \le \|\rho^n - \tilde \rho^n\|_{L^1(\hat \Omega)} + I_1 +\tilde I_1 + I_2+\tilde I_2,
\end{equation}
where
\begin{align*}
I_1  & = \frac85\dt\kappa \sum_K \sum_{L\sim K} \tau_{KL} |\rho_K^{n+1}-\rho_L^{n+1}| | w_{KL} - \tilde  w_{KL}|,\\
I_2 & = \frac23 \dt\kappa \sum_K \sum_{L \sim K} \tau_{KL} \rho_L^{n+1} |w_{KL}+\tilde w_{KL}| |w_{KL}-\tilde w_{KL}|,
\end{align*}
and $\tilde I_1$ and $\tilde I_2$ are defined analogously with $\rho^{n+1}$ replaced by $\tilde \rho^{n+1}$. Of course, it is enough to estimate $I_1$ and $I_2$ as the argument for the other two terms is exactly identical.

We start by noticing that 
\[
|w_{KL}-\tilde w_{KL}| \le \frac{d_{KL}}{4\kappa} \Lip(W) \left(\|\rho^{n+1}-\tilde \rho^{n+1}\|_{L^1(\hat \Omega)} + \|\rho^n-\tilde \rho^n\|_{L^1(\hat \Omega)}\right),
\]
where we made use of the Lipschitz assumption on the interaction potential in \eqref{A3}.
We thus have the bound
\[
I_1 \le \frac25 \dt \Lip(W) \sum_K \sum_{L\sim K} |K\edge L| |\rho_K^{n+1} - \rho_L^{n+1}| \left(\|\rho^{n+1}-\tilde \rho^{n+1}\|_{L^1(\hat \Omega)} + \|\rho^n-\tilde \rho^n\|_{L^1(\hat \Omega)}\right).
\]
We now invoke the gradient estimate in \eqref{fine_gradient_estimate} and the smallness assumption \eqref{smallness_condition} on the time step size to deduce that
\begin{equation}
\label{700}
I_1 \le \frac18 \|\rho^{n+1}-\tilde \rho^{n+1}\|_{L^1(\hat \Omega)} + \frac18 \|\rho^n-\tilde \rho^n\|_{L^1(\hat \Omega)}.
\end{equation}

Similarly, because $|w_{KL}|\le \frac{d_{KL}}{2\kappa} \Lip(W)$ by \eqref{A2}, we have that
\[
I_2\le \frac16 \frac{\dt}{\kappa}\Lip(W)^2 \sum_K \sum_{L\sim K} |K\edge L| d_{KL} \rho_L^{n+1} \left(\|\rho^{n+1}-\tilde \rho^{n+1}\|_{L^1(\hat \Omega)} + \|\rho^{n}-\tilde \rho^{n}\|_{L^1(\hat \Omega)} \right).
\]
By a relabelling argument, we deduce that
\[
I_2 \le \frac13 \frac{\dt}{\kappa}\Lip(W)^2\left(\|\rho^{n+1}-\tilde \rho^{n+1}\|_{L^1(\hat \Omega)} + \|\rho^{n}-\tilde \rho^{n}\|_{L^1(\hat \Omega)} \right),
\]
and applying \eqref{smallness_condition} again, we obtain
\begin{equation}
\label{701}
I_2  \le \frac18 \|\rho^{n+1}-\tilde \rho^{n+1}\|_{L^1(\hat \Omega)} + \frac18 \|\rho^n-\tilde \rho^n\|_{L^1(\hat \Omega)}.
\end{equation}
It remains to use \eqref{700} and \eqref{701} in \eqref{702} to deduce \eqref{e:stable} with $C=3$.
\end{proof}
Before turning to the proof of Theorem~\ref{thm:stat}, we list some auxiliary elementary properties of the function $\alpha_\kappa$  occurring in the definition of the dissipation~\eqref{e:def:dissipation} and following from the ones of $\theta_\kappa$ in Remark~\ref{lem:theta}.
\begin{remark}\label{lem:alpha}
  The function $\alpha_\kappa$ from~\eqref{e:def:alpha} satisfies the following properties:
  \begin{enumerate}[ (i) ]
   \item For any $v\in \R$ is $\alpha_\kappa(\cdot,\cdot;v)$  jointly convex on $\R_+\times \R_+$.
   \item It holds $\alpha_\kappa(a,b,r) = 0$ if and only if $v=-\kappa\log\frac{a}{b}$.
   \item It holds $\alpha_{\kappa}(a,b;0) = \kappa(a-b)\log\frac{a}{b}$.
   \item It holds $\lim_{\kappa\to 0} \kappa \alpha_{\kappa}(a,b;v)= a (v_+)^2 + b (v_-)^2$.
  \end{enumerate}
\end{remark}
\begin{proof}[Proof of Theorem~\ref{thm:stat}]
  Let $\rho$ be a solution to~\eqref{e:stat:fixed_point}. Then it holds for any $K\sim L$
  \begin{equation}
   \frac{\rho_K}{\rho_L} = \exp\biggl(-  \kappa^{-1} \sum_J \abs{J} \rho_J \bigl( W(x_K- x_J) - W(x_K-x_J)\bigr)\biggr) = \exp\bra*{- \kappa^{-1} d_{KL} q_{KL}[\rho]}.
  \end{equation}
  Hence, we have that $\theta_{\kappa}(\rho_K,\rho_L,;d_{KL} q_{KL}[\rho]) = 0$ by Remark~\ref{lem:theta}\eqref{rem3}, showing that~\eqref{thm:stat:2} implies~\eqref{thm:stat:1}. 
  
  Now, let $\rho$ be a solution to~\eqref{e:stat:def}. We introduce the notation 
  \[
  \Phi_K[\rho] = \frac{\exp\bra*{-\kappa^{-1} \sum_{J} \abs{J} \rho_J W(x_K-x_J)}}{\sum_{\tilde K} \abs{\tilde K} \exp\bra*{-\kappa^{-1} \sum_{J} \abs{J} \rho_J W(x_{\tilde K}-x_J)}} .
  \]
  Our goal is to show that $\rho$ solves the fixed point equation \eqref{e:stat:fixed_point}, i.e., $\rho_K = \Phi_K[\rho]$.  Notice first that by using the following auxiliary relations
  \[
    d_{KL} q_{KL}[\rho] = - \kappa \log \frac{\Phi_K[\rho]}{\Phi_L[\rho]} \quad\text{and}\quad \exp\bra*{\frac{d_{KL}}{2\kappa} q_{KL}[\rho]} = \sqrt{\frac{\Phi_L[\rho]}{\Phi_K[\rho]}},
  \]
  we may write for any $\set{h_K}_K$, similarly as in the proof of Lemma \ref{L2}, that
  \begin{align}
    \theta_\kappa\bra[\big]{ h_K \Phi_K[\rho],h_L\Phi_L[\rho]; d_{KL} q_{KL}(\rho)} &= d_{KL} q_{KL}[\rho] \frac{h_k \sqrt{\Phi_K[\rho] \Phi_L[\rho]} - h_L \sqrt{\Phi_K[\rho] \Phi_L[\rho]}}{\sqrt{\frac{\Phi_L[\rho]}{\Phi_K[\rho]}}-\sqrt{\frac{\Phi_K[\rho]}{\Phi_L[\rho]}}} \\
    &= - \kappa \bra*{h_K - h_L} \frac{\log \Phi_K[\rho] - \log \Phi_L[\rho]}{\frac{1}{\Phi_K[\rho]} - \frac{1}{\Phi_L[\rho]}} \\
    &= \kappa \bra*{h_K - h_L} \frac{1}{\Lambda\bra*{ \frac{1}{\Phi_K[\rho]}, \frac{1}{\Phi_L[\rho]}}} . \label{e:theta:identity}
  \end{align}
  Therefore, supposing that $\rho_K = h_K\Phi_K[\rho]$ for some reals $h_K$, which is possible thanks to the positivity of $\Phi_K[\rho]$, equation \eqref{e:stat:def} translates into 
  \[
  0 = \sum_{L\sim K} \tau_{KL} \frac{h_K-h_L}{\Lambda({\Phi_K[\rho]}^{-1},{\Phi_L[\rho]}^{-1})}.
  \]
Testing this equation by some $\varphi_K$ and symmetrizing the resulting sums again, we arrive at the weak formulation
  \[
    0 = \frac{1}{2} \sum_{K}\sum_{L\sim K} \tau_{KL}\frac{\bra*{\varphi_K - \varphi_L} \bra*{h_K - h_L}}{\Lambda\bra*{\Phi_K[\rho]^{-1},\Phi_L[\rho]^{-1}}} . 
  \]
  By choosing $\varphi_K = h_K$ and recalling that the logarithmic mean $\Lambda$ is positive because the $\Phi_K[\rho]$'s are positive, we observe that $\{h_K\}_K$ has to be constant, that means, $h_K=C$ for some $C\in \R$. In particular, it holds that $\rho_K = C \Phi_K[\rho]$, for any cell $K$, and since both $\{\rho_K\}_K$ and $\{\Phi_K[\rho]\}_K$ are probability distributions, we must have  $h_K=C=1$, which is want we aimed to prove. We have thus showed that~\eqref{thm:stat:1} and~\eqref{thm:stat:2} are equivalent.
 
  Now, we consider the variation of $\cF^h$ from~\eqref{e:free_energy:numerical} on $\cP(\cT^h)$. For this fix $s: \cT^h \to \R$ with $\sum_{K} |K| s_K = 0$ and consider the variation 
  \[
   \lim_{\eps\to 0} \frac{\cF^h(\rho+ \eps \, s ) - \cF^h(\rho)}{\eps} = \sum_K |K| D_K\cF^h(\rho) s_K 
  \]
  Hence, we can only identify $D\cF^h$ with a vector up to a constant $c\in \R$ as
 \[
   D_K\cF^h(\rho) = \kappa  \log \rho_K +  \sum_L |L|W(x_K-x_L) \rho_L + c . 
 \]
 Since, $\rho\in \cP(\cT^h)$, we immediately recover $c= \kappa\log Z^h(\rho)$ as in~\eqref{thm:stat:2}. Hence, each probability distribution $\rho$ on $\cP(\cT^h)$ satisfying~\eqref{e:stat:fixed_point} is a critical point of  $\cF^h$ and vise versa, showing the equivalence of~\eqref{thm:stat:2} and \eqref{thm:stat:3}.
 
 Likewise, we have from Remark~\ref{lem:theta} and Remark~\ref{lem:alpha} that $\alpha_\kappa(\rho_K,\rho_L,d_{KL} q_{KL}[\rho])=0$ if and only if $\theta_\kappa(\rho_K,\rho_L,d_{KL} q_{KL}[\rho])=0$, which implies by~\eqref{e:stat:def} that $\rho$ is a stationary solution if and only if $\cD^h(\rho)=0$, showing the equivalence of~\eqref{thm:stat:1} and \eqref{thm:stat:4}.
\end{proof}
\begin{proof}[Proof of Theorem~\ref{thm:longtime}]
  The proof of the large time behavior follows along a standard argument from the theory of dynamical systems (see for instance~\cite[Section 6]{Teschl}). We consider $\cP(\cT^h)\subset \ell^1(\cT^h)$ as a convex compact subset endowed with the $\ell^1(\cT^h)$ topology. 
  By the global well-posedness for the scheme in $\cP(\cT^h)$ from Theorem~\ref{thm:well-posed}, we can define the $\omega$-limit set for any $\rho^0\in \cP(\cT^h)$ given by
  \[
    \omega(\rho^0) =\set*{ \mu \in \cP(\cT^h) : \rho^{n_k} \to \mu \text{ in $\ell^1(\cT^h)$ for some subsequence } n_k \to \infty} .
  \]
  Since the scheme leaves $\cP(\cT^h)$ invariant, each positive orbit $\mathcal{O}^+(\rho^0) = \bigcup_{n\geq 0} \rho^{n}$ is compact in $\cP(\cT^h)$ and we find a convergent subsequence in $\cP(\cT^h)$, showing that $\omega$-limit is non-empty. 
  Since the scheme is by the stability estimate~\eqref{e:stable} in particular continuous in $\ell^1(\cT^h)$, we obtain by~\cite[Lemma 6.5]{Teschl} that any $\omega$-limit set  $\omega(\rho_0)$ is positive invariant, that is, for $\mu \in \omega(\rho^0)$ follows that $\mathcal{O}^+(\mu)\subseteq \omega(\rho^0)$. The compactness also implies by standard arguments that $\operatorname{dist}_{\ell^1(\cT^h)}(\rho^n , \omega(\rho^0))\to 0$ as $t\to \infty$ (see~\cite[Lemma 6.7]{Teschl}). 
   
  Hence, it is left to characterize the $\omega$-limit. To do so, we note that the free energy $\cF^h$ is bounded from below (by zero) on $\cP(\cT^h)$, thanks to the nonnegativity assumption in \eqref{A0}.  By the free energy dissipation relation~\eqref{e:FED:numerical} we have that $\cF^h(\rho^{n})$ is also monotone decreasing along the scheme, hence $\cF^h(\rho^n)\to \cF^\infty\in \R$. 
  In particular, for any $\overline\rho^0\in \omega(\rho^0)$ holds $\cF^h(\overline\rho^0)=\cF^\infty$. 
  Starting the scheme from any $\overline\rho^0\in \omega(\rho^0)$, we have from ~\eqref{e:FED:numerical} for any $N\in \N$ the identity
  \[
    \cF^\infty + \kappa \sum_{n=0}^N \cH(\overline\rho^{n+1}| \overline\rho^{n}) + \dt \sum_{n=1}^N \cD^h(\overline\rho^n) = \cF^\infty ,
  \]
  which implies that $\cH(\overline\rho^{n+1}| \overline\rho^{n})= 0 = \cD^h(\overline\rho^{n+1})$ for any $n\in \N_0$, since both, relative entropy and dissipation are nonnegative. Hence, $\omega(\rho^0)$ consists of elements $\overline\rho^0\in \cP(\cT^h)$ satisfying $\cD^h(\overline\rho^0)=0$, 
  implying that any $\overline\rho^0\in \omega(\rho^0)$ is a stationary state by Theorem~\ref{thm:stat}. 
\end{proof}

\subsection{Convergence of the scheme}
In this section, we finally turn to the proof of the convergence of the Scharfetter--Gummel scheme, Theorem \ref{thm:convergence}. In particular, we may assume throughout this subsection  that $h$ is small in the sense of \eqref{e:h_small}.

In the next step, we translate the gradient bounds into continuity estimates for the approximate solution.
\begin{lemma}[Translations]\label{prop:translations}
For any $\tau>0$, it holds that
\begin{equation}\label{translation_time}
\int_0^T \biggl(\int_{\Omega} |\rho_{h,\dt}(t+\tau,x) - \rho_{h,\dt}(t,x)|\dx{x}\biggr)^2 \dx{t} \lesssim \tau,
\end{equation}
uniformly in $h$. Moreover, for $\eta\in \R^d$, it holds that
\begin{equation}
\int_0^T \biggl( \int_{\R^d} |\rho_{h,\dt}(t,x+\eta) - \rho_{h,\dt}(t,x)|\dx{x}\biggr)^2 \dx{t}  = o(1),\label{translation_space}
\end{equation}
as $|\eta|\to0$, uniformly in $h$.
\end{lemma}
In the second statement of the lemma, we think of $\rho_{h,\dt}$ being extended trivially to all of $\R^d$, so that the spatial translations are all well-defined.  Notice that we lose an order of convergence in \eqref{translation_space} because of the presence of the boundary $\partial \Omega$. 
\begin{proof}
We start with the argument for the spatial translations estimated in \eqref{translation_space}. Notice that since $\rho_{h,\dt}$ is a probability measure on $\hat \Omega$, we may always  assume that $|\eta|\le 1$. We first split the integral in $x$ according 
\begin{align}
\MoveEqLeft \int_{\R^d} \bigl\lvert\rho_{h,\dt}(t,x+\eta) - \rho_{h,\dt}(t,x)\bigr\rvert\dx{x}\nonumber \\
 & = \int_{\{x\in\hat \Omega, \, x+\eta\in \hat \Omega\}} \bigl\lvert \rho_{h,\dt}(t,x+\eta) - \rho_{h,\dt}(t,x) \bigr\rvert \dx{x}
 +\int_{\{x\in\hat \Omega, \, x\pm\eta\not \in \hat \Omega\}} \rho_{h,\dt}(t,x)\dx{x}\label{301}
\end{align}
Let us start with the  treatment of the first integral term. Similarly as in \cite{EymardGallouetHerbin00}, we let $[x,z]$ be the line segment between two points $x$ and $z$ and define the characteristic function of neighboring cells $K\sim L$ by
\[
\chi_{KL}(x,z) =\begin{cases} 1&\text{, if }[x,z]\cap K\edge L\not=\emptyset,\\
0&\text{, else} \:. 
\end{cases}
\]
By geometric arguments, we observe that
\[
\int_{\hat \Omega} \chi_{KL}(x,x+\eta)\dx{x} \le |K\edge L| \, |\eta|.
\]
With the help of the triangle inequality, we thus estimate
\begin{align*}
\int_{\{x\in \hat \Omega,\, x+\eta\in \hat \Omega\}}\bigl|\rho_{h,\dt}(t,x+\eta) - \rho_{h,\dt}(t,x)\bigr|\dx{x} &  \le  \sum_K \sum_{L\sim K} |\rho_K^n-\rho_L^n| \int_{\hat \Omega}\chi_{KL}(x,x+\eta)\dx{x} \\
&\le |\eta| \sum_K\sum_{L\sim K} |K\edge L| \, |\rho_K^n-\rho_L^n|.
\end{align*}
Squaring both sides, integrating in time and invoking the gradient bound \eqref{spatial_gradient} together with the control of the initial datum in \eqref{ini_dat} yields an error of order $O(|\eta|)$. 

Considering the second term in \eqref{301}, we notice that
\[
\int_{\{x\in\hat \Omega, \, x\pm\eta\not \in \hat \Omega\}} \rho_{h,\dt}(t,x)\dx{x} \le 2\int_{\{x\in \hat \Omega:\: \dist(x,\partial\hat \Omega)\le |\eta|\}} \rho_{h,\dt}\dx{x}.
\]
To estimate the right-hand side, we first note, that it is enough to estimate the integral on the set $A_\eta =\set*{\dist(\cdot,\partial\hat\Omega) \le |\eta|} \cap \set{\rho_{h,\dt}(t,\cdot)\geq 1 }$, since on the complement the integral is bounded by $\abs*{A_\eta}\to 0$ as $\eta\to 0$. Now, we let $\varphi(z) = z\max\set*{\log z,1}$ and denote with $\varphi^{-1}(r) = \sup\set*{z : z\leq \varphi(r)}$. With this, we can apply Jensen's inequality
\begin{align*}
\int_{ A_\eta } \rho_{h,\dt}\dx{x} \le \bigl|A_\eta \bigr| \varphi^{-1}\biggl( \frac{1}{\abs*{A_\eta}} \int_{A_\eta}\varphi (\rho_{h,\dt})\dx{x}\biggr)
\end{align*}
Thanks to the bound on the free energy \eqref{e:FED:numerical} and the initial datum \eqref{ini_dat}, this estimate yields that
\[
\int_{\{x\in\hat \Omega, \, x\pm\eta\not \in \hat \Omega\}} \rho_{h,\dt}(t,x)\dx{x}  \le 2 \bigl|A_\eta\bigr| \varphi^{-1}\biggl(\frac{C}{\bigl|A_\eta\bigr|}\biggr),
\]
for some $C>0$, independent of $h$ and $\eta$. It is clear that $\varphi^{-1}$ is growing only sublinearily, and therefore, we have an $o(1)$ bound as $|\eta|\to 0$, uniformly in $h$.

The estimate on the temporal increments \eqref{translation_time} is obtained simply by  using \eqref{time_derivative}	together with the triangle inequality.
\end{proof}
Before proving the convergence of the scheme, we provide an auxiliary lemma on the convergence of discrete gradients, which we introduce in the following. For this, we define the \emph{diamonds}
\[
D_{KL} = \left\{t x_K+(1-t) y: t\in(0,1], y\in K\edge L\right\} \subset K,
\]
for any $K, L\in \cT$ with $L\sim K$. Of course, $K = \dot \cup_{L\sim K} D_{KL}\dot \cup \{x_K\}$ by construction. Since~$\cT^h$ is a Voronoi tesselation, we obtain that the volume of these cells can be computed according to the law
\begin{equation}\label{300}
|D_{KL}| = \frac{d_{KL}}{2d} |K\edge L|.
\end{equation}
For a sequence $\zeta_{h,\dt} = \{\zeta^n_K\}$ on $(0,T)\times \hat \Omega$, we may now introduce the discrete gradients
\begin{equation}\label{12}
\grad^h \zeta_{h,\dt}(t,x) = \frac{d}{d_{KL}} (\zeta^n_L-\zeta^n_K)\nu_{KL}\quad \text{for any }(t,x)\in (t^n,t^{n+1})\times D_{KL}.
\end{equation}
Here and in the following,  $\nu_{KL}$ denotes  the outer (with respect to $K$) unit normal vector on the edge $K\edge L$.

It is a well-known fact that these gradients are convergent.
\begin{lemma}[Convergence of discrete gradients \cite{EymardGallouet03}]\label{L3}
Let $\zeta\in C^{\infty}([0,T]\times \R^d)$ be given and set $\zeta_{h,\dt} =  \zeta_K^n = \zeta(t^n,x_K)$ in $(t^n,t^{n+1})\times K$ for any $n$ and any $K$. Then 
\[
\grad^h\zeta_{h,\dt} \to \grad \zeta \quad \text{weakly-$*$ in }L^{\infty}((0,T)\times \Omega).
\]
\end{lemma}
In \cite[Lemma 2]{EymardGallouet03}, an $L^2$ variant of this result was established and the present version follows from minor modifications of the original proof, which are omitted here.

With these preparations at hand, we are now in the position to derive the compactness of the sequence of approximate solutions. Here, the discrete gradient $\grad^h\rho_{h,\dt}$ is defined by \eqref{12}.
\begin{proposition}[Compactness]\label{prop:compactness}
There exists $\rho\in L^2((0,T);L^1(\Omega))$ with $\grad \rho \in L^1((0,T)\times \Omega)$, such that $\rho_{h,\dt} \to \rho$ strongly in $L^2((0,T);L^1(\Omega))$ and $\grad^h \rho_{h,\dt} \to \grad \rho$ weakly in $L^2((0,T);L^1( \Omega))$. 
\end{proposition}
\begin{proof}
The compactness of the sequence in $L^2((0,T);L^1(\Omega))$ is a consequence of Lemma~\ref{prop:translations} and the Riesz--Fr\'echet--Kolmogorov theorem. We now turn to the weak convergence of the gradients. The algebraic arguments are essentially identical to the convergence of the gradients in the smooth setting. We will thus skip those and focus  on the parts, which are different. 

We first show the convergence of the gradients of the roots $\grad^h\sqrt{\rho_{h,\dt}}$, for which we need the control on the Fisher information. In view of \eqref{300}, it is a short computation to obtain the relation
\[
\int_{\dt}^T \int_{\hat \Omega} |\grad^h \sqrt{\rho_{h,\dt}}|^2\dx{x}\dx{t}  = \frac{d}2 \dt\sum_{n=1}^{N} \sum_K \sum_{L\sim K}\frac{|K\edge L|}{d_{KL}} \left(\sqrt{\rho^n_K} - \sqrt{\rho^n_L}\right)^2.
\]
Thanks to \eqref{Fisher_information}, the right-hand side is bounded uniformly in $h$, and  thus, for every $\eps>0$, we find a function $f \in L^2((\eps,T)\times \Omega;\R^d)$ and a subsequence, not relabeled, both dependent on $\eps$, such that
\[
\grad^h \sqrt{\rho_{h,\dt}} \to f\quad\text{weakly in }L^2((\eps,T)\times \Omega).
\]
We have to show that $f = \grad{\sqrt{\rho}}$, and thus, the statement becomes independent of $\eps$, and holds true for the full sequence.

We choose a smooth test function $\varphi\in C_c^{\infty}((0,T)\times \Omega;\R^d)$ and suppose that $\eps$ is small in the sense that $\varphi(t,\cdot) = 0$ for any $t\le \eps$. We also assume that $\dt\le \eps$. Notice then that by algebraic reformulations similar to those in Lemma \ref{L3}, we have the two identities
\[
\int_0^T \int_{\Omega} \sqrt{\rho_{h,\dt}} \, \div \varphi\dx{x}\dx{t} = \sum_n \sum_K\sum_{L\sim K} \frac{d}{d_{KL}}\Bigl(\sqrt{\rho_K^n}-\sqrt{\rho^n_L}\Bigr)\nu_{KL}\cdot \int_{t^n}^{t^{n+1}}\int_{D_{KL}}\varphi^n_{KL}\dx{x}\dx{t},
\]
where $\varphi_{KL}^n$ is the average of $\varphi$ over $(t^n,t^{n+1})\times D_{KL}$, and
\[
- \int_0^T \int_{\Omega} \grad^h\!\sqrt{\rho_{h,\dt}} \cdot \varphi\dx{x}\dx{t} = \sum_n \sum_K\sum_{L\sim K} \frac{d}{d_{KL}}\Bigl(\sqrt{\rho_K^n}-\sqrt{\rho^n_L}\Bigr)\nu_{KL}\cdot \int_{t^n}^{t^{n+1}} \int_{D_{KL}}\varphi \dx{x}\dx{t}.
\]
Now, using $|\varphi - \varphi_{KL}^n| \lesssim h \|\varphi\|_{C^1}$ uniformly in $(t^n,t^{n+1})\times D_{KL}$, which holds true thanks to the smoothness of $\varphi$, and \eqref{300}, we find that
\begin{align*}
\MoveEqLeft \left|\sum_n\sum_K \sum_{L\sim K} \frac{d}{d_{KL}} \Bigl(\sqrt{\rho_K^n}-\sqrt{\rho_L^n}\Bigr) \nu_{KL} \cdot \int_{t^{n}}^{t^{n+1}} \int_{D_{KL}} (\varphi-\varphi_{KL}^n)\dx{x}\dx{t}\right|\\
& \lesssim h \|\varphi\|_{C^1} \dt \sum_n  \sum_K \sum_{L\sim K} |K\edge L|\,\Bigr|\sqrt{\rho_L^n}-\sqrt{\rho_K^n}\Bigl|\\
&\lesssim h \|\varphi\|_{C^1} \sqrt{T} \biggl(\dt \sum_n  \sum_K \sum_{L\sim K}\frac{|K\edge L|}{d_{KL}} \Bigl(\sqrt{\rho_K^n} - \sqrt{\rho_L^n}\Bigr)^2\biggr)^{\frac12}\biggl(\sum_K\sum_{L\sim K} d_{KL}|K\edge L|\biggr)^{\frac12}.
\end{align*}
By \eqref{Fisher_information} and the isoperimetric property \eqref{3}, the right-hand side vanishes as $h\to 0$. Therefore, we have 
\[
\int_0^T \int_{\Omega} \sqrt{\rho_{h,\dt}} \, \div \varphi\dx{x}\dx{t}  = - \int_0^T \int_{\Omega} \grad^h \sqrt{ \rho_{h,\dt}} \cdot \varphi\dx{x}\dx{t}+o(1) \qquad\text{ as } h\to 0 \:.
\]
We notice that $\sqrt{\rho_{h,\dt}}\to \sqrt{\rho}$ in $L^2((0,T)\times \Omega)$ thanks to the elementary inequality $(\sqrt{a}-\sqrt{b})^2\le |a-b|$.  Therefore, sending $h$ to zero, we thus obtain
\[
\int_0^T \int_{\Omega} \sqrt{\rho}\,\div \varphi\dx{x}\dx{t} = -\int_0^T\int_{\Omega}f\cdot  \varphi  \dx{x}\dx{t}.
\]
This statement implies that $f = \grad\!\sqrt\rho$ since $\varphi$ was arbitrary.

Furthermore, by the Cauchy--Schwarz inequality and the fact that $\rho$ is a probability distribution, we observe that
\[
\int_0^T\left(\int_{\Omega} |\grad \rho|\dx{x}\right)^2 \dx{t} \le \int_0^T  \int_{\Omega} \frac{|\grad\rho|^2}{\rho}\dx{x}\dx{t} = 4\int_0^T\int_{\Omega}\bigl|\grad \sqrt{\rho}\bigr|^2\dx{x}\dx{t},
\]
and thus $\grad \rho \in L^2((0,T);L^1(\Omega))$. 

In order to show the convergence of the gradients,  notice that thanks to the bound in~\eqref{spatial_gradient}, we have an estimate on the spatial gradient. Indeed, using  formula \eqref{300}, we observe that
\[
\int_{\dt}^T \left(\int_{\Omega} |\grad^h\rho_{h,\dt}|\dx{x}\right)^2 \dx{t}  = \frac{\dt}{4} \sum_{n=1}^N \left(\sum_K \sum_{L\sim K} |K\edge L| \, |\rho_L^n-\rho_K^n|\right)^2.
\]
It thus follows that, for every positive $\eps$, there exists a vector-valued measure $g \in L^2((\eps,T);\mathcal{M}( \Omega;\R^d))$ and  a subsequence (not relabeled), such that
\[
\grad^h \rho_{h,\dt} \to g\qquad \text{weakly-$*$ in }L^2((\eps,T);\mathcal{M}( \Omega)).
\]
If we know that $g = \grad \rho$,   the desired convergence result immediately follows because we have just seen that $\grad \rho\in L^2((0,T);L^1(\Omega))$. We skip the argument for this missing ingredient, as it closely resembles the previous part of the proof. Notice only that the gradient bound in \eqref{spatial_gradient} has to be used instead of the estimate on the Fisher information in \eqref{Fisher_information}.
\end{proof}
We are now in the position to prove the convergence of the scheme, which is the first statement of Theorem \ref{thm:convergence}. The convergence of the stationary solutions will be addressed afterwards.
\begin{proof}[Proof of Theorem \ref{thm:convergence}. Convergence of the scheme.]
The compactness of the sequence is established in Proposition \ref{prop:compactness}.
Notice also that the sequence of the initial data $\rho_{h}^0$ converges in $L^1(\Omega)$ towards $\rho_0$ by construction. It remains thus to prove that the limit function $\rho$ solves the continuous equation \eqref{e:agg-diff}. For this purpose, we choose a subsequence of time steps $\dt$ such that   $N\dt = T$ for some integer $N$, consider a test function $\zeta\in C_c^{\infty}([0,T)\times \R^d)$, and suppose that $\dt$ is sufficiently small such that  $\zeta(t)=0$ for $t\ge T -2 \dt$.  We set $\zeta^n_K = \zeta(t^n,x_K)$, and write $\zeta_{h,\dt} = \zeta^n_K$ in $(t^n,t^{n+1})\times K$. 

Testing the discrete equation \eqref{e:time:flux} with $
\dt\, \zeta_K^n$ and using the flux identity \eqref{e:flux:identity2} gives
\begin{align*}
 0 & = \sum_{n=0}^{N-1} \sum_K |K| \,\zeta_K^n \, \bigl(\rho_K^{n+1}-\rho_K^n\bigr) + \dt \sum_{n=0}^{N-1}\sum_K \sum_{L\sim K}  |K\edge L|  \, \zeta_K^n \, q_{KL}^{n+1}\frac{\rho_K^{n+1}+\rho_L^{n+1}}2\\
&\qquad  +  \frac{\dt}{2} \sum_{n=0}^{N-1}\sum_K \sum_{L\sim K} |K\edge L|  \, \zeta_K^n \, q_{KL}^{n+1}\coth\biggl(\frac{d_{KL}q_{KL}^{n+1}}{2\kappa}\biggr)\bigl(\rho_K^{n+1}-\rho_L^{n+1}\bigr)\\
&= I^h_1 + I^h_2 + I^h_3.
\end{align*}
\noindent\emph{Convergence of $I_1^h$}:  We first turn to the convergence of the term that involves the time derivative. Performing a discrete integration by parts in the time variable, we have that
\begin{align*}
I_1^h & = -\sum_{n=1}^{N-1} \sum_K |K| (\zeta_K^n-\zeta_K^{n-1}) \rho_K^n -\sum_K |K|\zeta_K^0\rho_K^0 + \sum_K |K| \zeta_K^{N-1}\rho_K^N.
\end{align*}
Notice that by our choice of $\zeta$, it holds that $\zeta_K^{N-1} =\zeta(t^{N-1},x_K) = \zeta((N-1)\dt,x_K)=0$, and thus, the last term on the right-hand side vanishes. For the middle term, we observe that $\zeta_K^0 = \zeta(0,x_K) = \zeta(0,x) + O(|x-x_K|)$ for any $x\in K$ by the smoothness of the test function, and thus, because $\rho_0$ is a probability distribution,
\[
\sum_K |K| \zeta_K^0 \rho_K^0 = \sum_K \int_K \zeta(0,x)\rho_0(x)\dx{x} + O(h) = \int_{ \Omega}\zeta(0,x)\rho_0(x)\dx{x} + O(h).
\]
Finally, using the smoothness of $\zeta$ again, we expand $\zeta_K^n-\zeta^{n-1}_K = \partial_t\zeta(t^{n-1},x_K)\dt +O(\dt)^2  = \partial_t\zeta(t,x)\dt +O(\dt\, h)$ for any $(t,x)\in (t^n,t^{n+1})\times K$, and thus, similarly as before,
\begin{align*}
-\sum_{n=1}^{N-1}\sum_K |K|(\zeta_K^n-\zeta_K^{n-1})\rho_K^n& = -\sum_{n=1}^{N-1}\int_{t^n}^{t^{n-1}} \sum_K \int_K \frac{\zeta_K^n-\zeta_K^{n-1}}{\dt} \rho_{h,\dt}(t,x)\dx{x}\dx{t}\\
& = -\int_{\dt}^{T} \int_{\hat \Omega} \partial_t \zeta(t,x)\rho_{h,\dt}(t,x)\dx{x}\dx{t} +O(h).
\end{align*}
Because $\rho_{h,\dt}$ is converging stongly in $L^1((0,T)\times \Omega)$ and because
\begin{equation}\label{302}
\int_{\hat \Omega\setminus\Omega} \rho_{h,\dt}\dx{x} = \int_{\Omega} (\rho-\rho_{h\,dt})\dx{x},
\end{equation}
since both $\rho$ and $\rho_{h,\dt}$ are probability distributions, we find that
\[
I_1^h \to - \int_0^T \int_{\Omega}\partial_t \zeta \, \rho\dx{x}\dx{t} -\int_{\Omega} \zeta(0,x)\rho_0(x)\dx{x} , \qquad\text{as } h\to0 \:. 
\]
\noindent\emph{Convergence of $I_2^h$}: By a discrete integration by parts, we have
\begin{align*}
I_2^h & =\frac{\dt}{2} \sum_n  \sum_K \sum_{L\sim K} |K\edge L| \bigl(\zeta_K^n-\zeta_L^n\bigr) q_{KL}^{n+1}\frac{\rho_K^{n+1}+\rho_L^{n+1}}2\\
& = \frac{\dt}{2} \sum_n \sum_K \sum_{L\sim K}|K\edge L| \bigl(\zeta_K^n-\zeta_L^n\bigr) q_{KL}^{n+1}\rho_K^{n+1}\\
&\qquad + \frac{\dt}{2} \sum_n \sum_K \sum_{L\sim K}|K\edge L| \bigl(\zeta_K^n-\zeta_L^n\bigr)q_{KL}^{n+1}\frac{\rho_L^{n+1}-\rho_K^{n+1}}2.
\end{align*}
As a consequence of the Lipschitz property of the interaction potential \eqref{A3} and the fact that the numerical solution is a probability distribution, it holds that $|q_{KL}^{n+1}|\le \Lip(W)$. Moreover, by the smoothness of $\zeta$, we have that $|\zeta_K^n-\zeta_L^n|\le 2h \|\grad\zeta\|_{C^0}$, and thus
\begin{align*}
\MoveEqLeft \biggl|\dt\sum_n  \sum_K \sum_{L\sim K} |K\edge L| \bigl(\zeta_K^n-\zeta_L^n\bigr) q_{KL^{n+1}} \bigl(\rho_L^{n+1}-\rho_K^{n+1}\bigr)\biggr|\\
& \lesssim h \|\grad \zeta\|_{C^0}\Lip(W) \sum_n \dt \sum_K |K\edge L| \, \bigl|\rho_L^{n+1}-\rho_K^{n+1}\bigr|,
\end{align*}
and the right-hand side is of order $O(h)$ by the virtue of the gradient estimate in \eqref{spatial_gradient}.  Therefore,
\[
I_2^h    = \frac{\dt}{2} \sum_n  \sum_K \sum_{L\sim K}|K\edge L| \bigl(\zeta_K^n-\zeta_L^n\bigr) q_{KL}^{n+1}\,\rho_K^{n+1} +o(1) \qquad\text{ as } h\to0 . 
\]
Similarly for the time variable, we estimate slightly more carefully $|\zeta_K^n-\zeta_L^n|\le d_{KL}\|\grad\zeta\|_{C^0}$ and use  identity \eqref{300} to arrive at
\begin{align*}
\MoveEqLeft \biggl| \frac12 \sum_n \sum_K  \sum_{L\sim K}\int_{t^n}^{t^{n+1}} \int_{D_{KL}\cap \Omega} \bigl(\rho_{h,\dt}(t+\dt,x)-\rho(t,x)\bigr)\dx{x}\dx{t} \,  \frac{|K\edge L|}{|D_{KL}|}(\zeta_K^n-\zeta_L^n)\, q_{KL}^{n+1}\biggr| \\
&\le d \|\grad \zeta\|_{C^0}\Lip(W) \int_0^T    \int_{\Omega} \bigl|\rho_{h,\dt}(t+\dt,x) - \rho(t,x)\bigr|\dx{x}\dx{t} .
\end{align*}
In a similar manner, by the virtue of \eqref{302}, it holds that 
\begin{align*}
\MoveEqLeft \left| \frac12 \sum_n \sum_K  \sum_{L\sim K}\int_{t^n}^{t^{n+1}} \int_{D_{KL}\setminus \Omega}  \rho_{h,\dt}(t+\dt,x) \dx{x}\dx{t} \, \frac{|K\edge L|}{|D_{KL}|}(\zeta_K^n-\zeta_L^n) q_{KL}^{n+1}\right| \\
&\le d \|\grad \zeta\|_{C^0}\Lip(W) \int_0^T    \int_{\Omega} \bigl|\rho_{h,\dt}(t+\dt,x) - \rho(t+\dt,x)\bigr|\dx{x}\dx{t} .
\end{align*}
As a consequence of the strong convergence established above and the continuity in time of approximate solutions \eqref{translation_time}, we see that in both estimates the right-hand side vanishes as $h$ and $\dt$ converge to 0.  Using \eqref{300} again, we thus conclude  that
\[
I_2^h =  \sum_n \sum_K  \sum_{L\sim K} \int_{t^n}^{t^{n+1}} \int_{D_{KL}\cap \Omega} \rho(t,x)\dx{x}\dx{t}  \frac{d}{d_{KL}} (\zeta_K^n-\zeta_L^n) q_{KL}^{n+1} +o(1), \qquad\text{as } h\to 0 \:.
\]
To estimate the flux term, we let $\eps>0$ be an arbitrarily fixed number and suppose that $h$ is small such that $h<\eps$. We start by noticing that, because $W$ is differentiable away from the origin by \eqref{A2}, we have the expansion
\[
W(x_K-x_J) = W(x_L-x_J) + \grad W(x_L-x_J)\cdot (x_K-x_L) + o(d_{KL}),
\]
for any cell $J$ that satisfies $\dist(J,x_K)\ge 2\eps$ and then also $\dist(J,x_L)\ge \eps$. 
Thus, using the continuity of $\grad W$ away from the origin, cf.~\eqref{A2}, as $d_{KL}\le 2h $, we have for any $x \in K$ and $y \in J$ that
\[
\frac{W(x_K-x_J) - W(x_L-x_J)}{d_{KL}} = \grad W(x-y) \cdot \frac{x_K-x_L}{|x_K-x_L|} +o(1), \qquad\text{as } h\to 0 \:.
\]
Therefore, we have to leading order, using in addition the Lipschitz property of $W$ in \eqref{A3} and the fact that $\rho_{h,\dt}$ is a probability distribution,
\begin{align*}
\MoveEqLeft \biggl| q_{KL}^{n+1} +\sum_{J} \frac{\rho^{n+1}_J + \rho_J^n}{2} \int_J \grad W(x-y)\dx{y}\cdot \nu_{KL} \biggr|\\
&\le 2 \Lip(W) \sum_{J\text{ s.t.\ }\dist(J,x_K)\le 2\eps} |J|\frac{\rho_J^{n+1}+\rho_J^n}2  +o(1)\\
&\le \Lip(W) \int_{B_{4\eps}(x_K)} \bigl( \rho_{h,\dt}(t+\dt,y)+\rho_{h,\dt}(t,y) \bigr) \dx{y} +o(1),
\end{align*}
for any $t\in [t^n,t^{n+1}]$ and any fixed $\eps$, as $h\to0$. It remains to replace the approximate solution in the convolution integral. Doing so, we obtain for any $t\in [t^n,t^{n+1}]$ and $x\in K$,
\begin{align*}
\MoveEqLeft\biggl|q^{n+1}_{KL} + \int_{\Omega}\rho(t,y)\grad W(x-y)\dx{y} \cdot \nu_{KL}\biggr|\\
& \lesssim \Lip(W) \int_{\Omega} \bigl|\rho_{h,\dt}(t+\dt,y) +\rho_{h,\dt}(t,y)-2\rho(t,y)\bigr|\dx{y}\\
&\quad + \Lip(W) \int_{\hat \Omega\setminus\Omega} \bigl(\rho_{h,\dt}(t+\dt,y)+\rho_{h,\dt}(t,y)\bigr)\dx{y}\\
&\quad+\Lip(W)  \int_{B_{4\eps}(x_K)\cap \Omega} \bigl( \rho_{h,\dt}(t+\dt,y)+\rho_{h,\dt}(t,y)\bigr)\dx{y} +o(1),
\end{align*}
as $h\to 0$. Thanks to the continuity in time \eqref{translation_time},  the previously established convergence of the approximating sequence, and identity \eqref{302}, and using $|\zeta_K^n-\zeta_L^n| \le d_{KL} \|\grad \zeta\|_{C^0}$ again, we arrive at
\begin{align*}
\MoveEqLeft \biggl| I_2^h -  \sum_n \sum_K \sum_{L\sim K} \int_{t^n}^{t^{n+1}} \int_{D_{KL}\cap \Omega} \rho(t,x) (\grad W\ast\rho)(t,x)\dx{x}\dx{t} \cdot  \frac{d}{d_{KL}} (\zeta_L^n-\zeta_K^n) \nu_{KL} \biggr|\\
&\lesssim \Lip(W) \int_0^T \int_{B_{4\eps}(x_K)\cap \Omega}\rho \dx{y}\dx{t} +o(1), \qquad\text{as } h\to 0 \:.
\end{align*}
Alternatively, using the definition of the discrete gradient in \eqref{12}, this estimate can be written as
\begin{align*}
 \left| I_2^h-\int_0^{T }  \int_{\Omega} \rho  (\grad W\ast \rho)\cdot \grad^h\zeta_{h,\dt}\dx{x}\dx{t} \right|\lesssim \Lip(W)  \int_0^T\int_{B_{4\eps}(x_K)\cap \Omega}\rho \dx{y} \dx{t}+o(1),
\end{align*}
as $h\to 0$. Apparently, the ($h$-indenpendent) first term on the right-hand side vanishes as $\eps\to0$. Therefore, because $\rho \grad W\ast \rho\in L^1((0,T)\times \Omega)$ as a consequence of \eqref{A3}, we may invoke Lemma \ref{L3} and find that
\[
I_2^h \to \int_0^T \int_{\Omega} \rho(\grad W\ast \rho)\cdot \grad \zeta\dx{x}\dx{t}, \qquad\text{as } h\to 0 \:.
\]
\noindent\emph{Convergence of $I_3^h$}: Again, we start with a discrete integration by parts to rewrite $I_3^h$,
\[
I_3^h = \frac{\dt}{4} \sum_n \sum_K \sum_{L\sim K} |K\edge L| \bigl(\zeta_K^n-\zeta_L^n\bigr) \, q_{KL}^{n+1}\coth\biggl(\frac{d_{KL}q_{KL}^{n+1}}{2\kappa}\biggr) \bigl(\rho_K^{n+1} - \rho_L^{n+1}\bigr).
\]
In order to get rid of the nonlinearity, we notice that the function $s\mapsto s\coth(s)$ is regular at the origin and $s\coth(s) = 1 +O(s^2)$ as $s\to 0$. Then, by using the bounds $|q_{KL}^{n+1}|\le \Lip(W)$ and $|\zeta_K^n-\zeta_L^n|\le 2h \|\grad \zeta\|_{C^0}$ as before, we have that the second order contribution is bounded by
\begin{align*}
\MoveEqLeft \frac{\dt}{\kappa}\sum_n \sum_K \sum_{L\sim K}|K\edge L| \, |\zeta_K^{n}-\zeta_L^{n}|\, d_{KL} \, |q_{KL}^{n+1}|^2\, |\rho_K^{n+1}-\rho_L^{n+1}| \\
&\lesssim \frac{h^2 \Lip(W)^2 \dt}{\kappa} \|\grad \zeta\|_{C^0}  \sum_n \sum_K \sum_{L\sim K} |K\edge L| \, |\rho_K^{n+1}-\rho_L^{n+1}| \\
&\leq \frac{h^2 \Lip(W)^2}{\kappa^2} \|\grad \zeta\|_{C^0}  \biggl(\kappa^2 \dt \, T  \sum_n \Bigl( \sum_K \sum_{L\sim K} |K\edge L| \, |\rho_K^{n+1}-\rho_L^{n+1}| \Bigr)^2\biggr)^{\frac{1}{2}} 
\end{align*}
and the right-hand side converges to $0$ as $h\to 0$ by the virtue of \eqref{spatial_gradient}. Therefore, we have to leading order
\[
I_3^h  = \frac{\kappa \dt}2 \sum_n\sum_K \sum_{L\sim K} \frac{|K\edge L|}{d_{KL}} (\zeta_K^n-\zeta_L^n) (\rho_K^{n+1}-\rho_L^{n+1}) +o(1), \qquad\text{as } h\to 0 . 
\]
Using the regularity of $\zeta$ again, we find that $\zeta_K^n-\zeta_L^n = \grad\zeta(t,x) \cdot(x_k-x_L) +O(h d_{KL})$ for any $(t,x) \in (t^{n+1},t^{n+2})\times D_{KL}$, where we have used that $d_{KL},\dt\lesssim h$. Here, the time interval $(t^{n+1},t^{n+2})$ is chosen such that the gradients $\nabla\zeta$ and $\nabla^h\rho_{h,\dt}$ are evaluated at the same time points. Indeed, we observe from~ \eqref{300}, that it holds 
\begin{align*}
\MoveEqLeft \Biggl|\frac{\kappa \dt}2 \sum_n \sum_K \sum_{L\sim K} \frac{|K\edge L|}{d_{KL}} (\zeta_K^n-\zeta_L^n)(\rho_K^{n+1}-\rho_L^{n+1}) \\
& + \kappa d \sum_n \sum_K\sum_{L\sim K}\int_{t^{n+1}}^{t^{n+2}} \int_{D_{KL}} \grad \zeta(t,x)\dx{x}\dx{t}\cdot \nu_{KL}\frac{\rho_K^{n+1}-\rho_L^{n+1}}{d_{KL}}\Biggr|\\
&\qquad \lesssim h\kappa \dt \sum_n \sum_K\sum_{L\sim K} |K\edge L|\,|\rho_K^{n+1}-\rho_L^{n+1}| \,,
\end{align*}
and the right-hand side vanishes as $h\to0$ thanks to the gradient estimate in \eqref{spatial_gradient}. Therefore, using the definition of the discrete gradient in \eqref{12},
\begin{align*}
I_3^h  & = - \kappa d \sum_n \sum_K\sum_{L\sim K}\int_{t^{n+1}}^{t^{n+2}} \int_{D_{KL}} \grad \zeta(t,x)\dx{x}\dx{t}\cdot \nu_{KL}\frac{\rho_K^{n+1}-\rho_L^{n+1}}{d_{KL}} +o(1)\\
& = \kappa \int_{\dt}^T \int_{\hat \Omega} \grad\zeta(t,x)\cdot \grad^h\rho_{h,\dt}(t,x)\dx{x}\dx{t} +o(1) , \qquad\text{as } h\to 0 . 
\end{align*}
As the final step, we implement the gradient convergence established in Proposition \ref{prop:compactness}, and conclude that
\[
I_3^h \to \kappa \int_0^T \int_{\Omega} \grad \zeta\cdot \grad \rho\dx{x}\dx{t} , \qquad\text{as } h\to 0 . 
\] 
\noindent\emph{Summary.} Putting together the convergence results for $I_1^h$, $I_2^h$ and $I_3^h$, passing to the limit in the identity $I_1^h + I_2^h +I_3^h=0$ yields
\begin{align*}
\MoveEqLeft - \int_0^T \int_{\Omega}\partial_t \zeta \rho\dx{x}\dx{t}    + \int_0^T \int_{\Omega} \rho(\grad W\ast \rho)\cdot \grad \zeta\dx{x}\dx{t} +\kappa \int_0^T \int_{\Omega} \grad \zeta\cdot \grad \rho\dx{x}\dx{t}\\
& =\int_{\Omega} \zeta(0,x)\rho^0(x)\dx{x},
\end{align*}
for any test function $\zeta$. This is the distributional formulation of the aggregation-diffusion equation.
\end{proof}
Similar to the convergence of the scheme, the convergence of stationary solutions is based on the Riesz--Fr\'echet--Kolmogorov compactness theorem. As a preparation, we derive estimates on discrete gradients.
\begin{lemma}\label{L10}
Let $\{\rho_K\}_K$ be a stationary solution of the Scharfetter--Gummel scheme. Then it holds
\begin{equation}\label{400}
\sum_K\sum_{L\sim K}|K\edge L|\, |\rho_K-\rho_L| \lesssim \frac{\Lip(W)}{\kappa}.
\end{equation}
\end{lemma}
\begin{proof}We recall that by Theorem \ref{thm:stat}, the stationary solutions obeys the equation
\[
\rho_K = \frac{1}{Z(\rho)} \exp\biggl(-\frac1{\kappa}\sum_J|J| \, W(x_K-x_J)\rho_J\biggr).
\]
Making use of the elementary estimate $|\exp(a) - \exp(b)| \le \bigl(\exp(a)+\exp(b)\bigr)\, |a-b|$, we thus have that
\[
|\rho_K-\rho_L| \le\frac1{\kappa} \left(\rho_K+\rho_L \right)\sum_J |J|\, \bigl|W(x_K-x_J)-W(x_L-x_J)\bigr|\rho_J.
\]
Because the aggregation potential is Lipschitz \eqref{A3}, and $\{\rho_K\}_K$ a probability distribution, we find
\[
|\rho_K-\rho_L| \le \frac{d_{KL}\Lip(W)}{\kappa} (\rho_K+\rho_L).
\]
It remains to apply a relabeling argument, the isoperimetric property \eqref{3} of the scheme, and, again, the fact that $\{\rho_K\}_K$ is a probability distribution to deduce that
\[
\sum_K\sum_{L\sim K}|K\edge L|\, |\rho_K-\rho_L| \lesssim \frac{\Lip(W)}{\kappa} . \qedhere
\]
\end{proof}
We now define the finite volume approximation $\rho_h$ of the stationary state $\{\rho_K\}_K$ by setting
\[
\rho_h = \rho_K\quad\text{in }K.
\]
We provide a continuity result. For this, we extend $\rho_h$ trivially to all of $\R^d$.
\begin{lemma}\label{L11}
For  $\eta\in \R^d$, it holds that
\begin{equation}\label{402}
  \int_{\R^d} |\rho_{h}(x+\eta) - \rho_{h}(x)|\dx{x}   = o(1),
\end{equation}
as $|\eta|\to0$, uniformly in $h$.
\end{lemma}
\begin{proof}
We first notice that stationary states have finite entropy, more precisely,
\begin{equation}\label{401}
\sum_K |K| \rho_K\log\rho_K \lesssim 1,
\end{equation}
uniformly in $h$. Indeed, in view of the properties \eqref{A0} and   \eqref{A3} of the potential and the fact that $\{\rho_K\}_K$ is a probability distribution, it holds that
\[
  \sum_J |J|\, |W(x_K-x_J)| \rho_J \le \Lip(W) \diam(\hat \Omega) \sum_J |J|\rho_J=\Lip(W) \diam(\hat \Omega).
\]
Therefore, invoking the characterization of stationary states in Theorem \ref{thm:stat}, we deduce that
\begin{align}\label{403}
\kappa |\log\rho_K| \le 2 \Lip(W) \diam(\hat \Omega) +  \log |\hat \Omega| .
\end{align}
Using again that $\{\rho_K\}_K$ is a probability distribution, \eqref{401} follows immediately.

The proof of \eqref{402} now follows almost identical to the one in the time-dependent setting, \eqref{translation_space}. We thus omit the details.
\end{proof}
We finally provide the argument for the second statement of Theorem \ref{thm:convergence}.
\begin{proof}[Proof of Theorem \ref{thm:convergence}. Convergence of stationary states.]
We first remark that as a consequence of Lemma \ref{L11} and the Riesz--Fr\'echet--Kolmogorov theorem, the sequence $\{\rho_h\}_h$ is compact in $L^1(\Omega)$, and thus, there exists $\rho \in L^1(\Omega)$ such that $\rho_h\to\rho$ in $L^1(\Omega)$, as $h\to0$. Our goal is to show that $\rho$ is a stationary solution of the aggregation-diffusion equation~\eqref{e:agg-diff}.

As in the previous two lemmas, our starting point is the characterization of stationary states in Theorem \ref{thm:stat},
\begin{equation}\label{404}
\kappa \log \rho_K =  - \sum_L|L|\, W(x_K-x_L) \, \rho_L + \log Z^h(\rho_h).
\end{equation}
We have seen in the proof of Lemma \ref{L11} that the right-hand side of this identity is bounded uniformly in $h$, see \eqref{403}, therefore, by the just established $L^1$ convergence and the dominated convergence theorem, it follows that
\[
\|\kappa \log\rho_h - \kappa \log \rho\|_{L^1(\Omega)} \to 0\qquad\text{as }h\to 0.
\]
Regarding the convergence of the right-hand side of \eqref{404}, we notice that, for any $x\in K$, it holds
\begin{align*}
 \biggl|\sum_L |L|W(x_K-x_L)\rho_L -(W\ast\rho)(x)\biggr|
 & \le \sum_L \int_{L\cap\Omega} |W(x_K-x_L)| \, |\rho_L-\rho(y)|\dx{y}\\
 &\qquad +\sum_L \int_{L\cap\Omega} |W(x_K-x_L)-W(x-y)|\rho(y)\dx{y}\\
 &\qquad  +\sum_L \int_{L\setminus \Omega} |W(x_K-x_L)|\rho_L\dx{y},
 \end{align*}
 and thus, using the properties \eqref{A0} and \eqref{A3} of the aggregation potential and identity \eqref{302}, which also holds true in the stationary case, we obtain that
 \begin{align*}
\MoveEqLeft{  \biggl|\sum_L |L|\, W(x_K-x_L)\,\rho_L -(W\ast\rho)(x)\biggr|}\\
& \le 2\Lip(W) \diam(\Omega) \|\rho_h-\rho\|_{L^1(\Omega)} +h\Lip(W) \|\rho\|_{L^1(\Omega)} = o(1),
  \end{align*}
as $h\to0$. We easily deduce that
\[
\sum_K \int_K \biggl|\sum_L |L|W(x_K-x_L)\rho_L -(W\ast\rho)(x)\biggr|\dx{x} \to 0\qquad\text{as }h\to 0,
\]
or, in words, the first term on the right-hand side of \eqref{404} converges to $- W\ast\rho$ in $L^1(\Omega)$. 

We consider now a mean-zero test function $\varphi\in C_c^{\infty}(\Omega)$ and define its (mean-zero) finite-volume approximation, as usual, by $\varphi_h = \varphi_K = \frac{1}{\abs*{K}}\int_K\varphi\dx{x}$ on $K$. Then $\varphi_h\to\varphi$ uniformly on $\Omega$. Testing \eqref{404} by $\varphi_h$ then yields
\[
\sum_K |K| \left(\kappa \log\rho_K + \sum_L|L|W(x_K-x_L)\rho_L\right)\varphi_K=0,
\]
and passing to the limit $h\to0$ gives
\[
\int_{\Omega} \left(\kappa \log \rho + W\ast\rho\right)\varphi\dx{x} = 0,
\]
which is, since $\varphi$ was arbitrary and mean-zero, equivalent to 
\[
\kappa \log \rho + W\ast\rho  + \kappa \log Z(\rho)=0,
\]
where $Z(\rho) =  \int_{\Omega} \exp(-\kappa^{-1}W\ast\rho)\dx{x}$. This is a characterization of stationary solutions for the 
 continuous problem as can be verified by inspection of the free energy dissipation~\eqref{e:FED}.
\end{proof}

\section*{Acknowledgement}
The authors thank the anonymous referees, Mario Ohlberger and Ansgar Jüngel for fruitful comments on an earlier version of the manuscript.
This work is funded by the Deutsche Forschungsgemeinschaft (DFG, German Research Foundation) under Germany's Excellence Strategy EXC 2044 -- 390685587, \emph{Mathematics M\"unster: Dynamics--Geometry--Structure,} and EXC 2047 -- 390685813, \emph{Hausdorff Center for Mathematics},
as well as the Collaborative Research Center 1060 -- 211504053, \emph{The Mathematics of Emergent Effects} at the Universität Bonn. 

\bibliographystyle{abbrv}
\bibliography{scharfetter-gummel}
\end{document}